\documentclass[12pt]{amsart}

\usepackage{youngtab}

\usepackage[latin1]{inputenc}
\usepackage{amsxtra}
\usepackage{amsmath}
\usepackage{amscd}
\usepackage{amssymb}
\usepackage{amsfonts}
\usepackage[all]{xy}
\usepackage{mathrsfs}
\usepackage{amsthm}
\usepackage{color}
\usepackage{upgreek}
\usepackage{verbatim}
\usepackage{enumerate}
\usepackage{latexsym}
\usepackage{graphicx}
\usepackage{tikz}
\usepackage{todonotes}
\usepackage{setspace}
\usepackage{hyperref}
\usepackage{stmaryrd}
\usepackage{nicefrac}
\usepackage{euscript}
\usetikzlibrary{decorations.pathmorphing}

 \definecolor{darkgreen}{HTML}{336633}
 \definecolor{darkred}{HTML}{993333}
\makeatletter

\newcommand{\arxiv}[1]{\href{http://arxiv.org/abs/#1}{\tt
    arXiv:\nolinkurl{#1}}}

\theoremstyle{plain}
\newtheorem{thm}{Theorem}[section]
\newtheorem{lem}[thm]{Lemma}
\newtheorem{prop}[thm]{Proposition}

\newtheorem{cor}[thm]{Corollary}
\newtheorem{df-prop}[thm]{Definition-Proposition}

\theoremstyle{definition}
\newtheorem{df}[thm]{Definition}

\theoremstyle{remark}
\newtheorem{rem}[thm]{Remark}
\newtheorem{ex}[thm]{Example}

\xyoption{all}
\newcommand{\tto}{\twoheadrightarrow}

\def\frakb{\mathfrak{b}}

\def\Hom{\operatorname{Hom}\nolimits}

\def\Res{\operatorname{Res}\nolimits}
\def\Ind{\operatorname{Ind}\nolimits}

\def\ep{\epsilon}
\def\gl{\mathfrak{gl}}

\def\la{\lambda}

\def\pn{\mf{pe} (n)}
\def\ov{\overline}

\newcommand{\mc}{\mathcal}
\newcommand{\mf}{\mathfrak}
\newcommand{\C}{\mathbb C}
\newcommand{\oo}{{\ov 0}}
\newcommand{\oa}{{\bar 0}}
\newcommand{\ob}{{\bar 1}}
\newcommand{\vare}{\epsilon} 

\newcommand{\ad}{\mathrm{ad}}
\newcommand{\cP}{\mathcal{P}}
\newcommand{\cF}{\mathcal{F}}
\newcommand{\fk}{\mathfrak{k}}
\newcommand{\fg}{\mathfrak{g}}
\newcommand{\fa}{\mathfrak{a}}
\newcommand{\fz}{\mathfrak{z}}
\newcommand{\fs}{\mathfrak{s}}
\newcommand{\fb}{\mathfrak{b}}
\newcommand{\fc}{\mathfrak{c}}
\newcommand{\fh}{\mathfrak{h}}
\newcommand{\fn}{\mathfrak{n}}
\newcommand{\n}{\mathfrak{n}}
\newcommand{\mZ}{\mathbb{Z}}
\newcommand{\cO}{\mathcal{O}}
\newcommand{\cA}{\mathcal{A}}
\newcommand{\cC}{\mathcal{C}}
\newcommand{\mC}{\mathbb{C}}
\newcommand{\mR}{\mathbb{R}}
\newcommand{\h}{\mathfrak{h}}
\newcommand{\N}{\mathbb N}
\newcommand{\ch}{\mathrm{ch}}
\newcommand{\rad}{\mathrm{rad}}

\newcommand{\Coind}{{\rm Coind}}

\newcommand{\LT}{{T_{w_0}}}
\newcommand{\g}{\mathfrak{g}}
\newcommand{\fl}{\mathfrak{l}}
\newcommand{\fp}{\mathfrak{p}}
\newcommand{\fu}{\mathfrak{u}}
\newcommand{\Real}{\mathrm{Re}}
\newcommand{\RP}{\mathsf{RP}}
\newcommand{\BRP}{\mathsf{BRP}}
\newcommand{\mN}{\mathbb{N}}
\newcommand{\Par}{\mathrm{Par}}
\newcommand{\Bor}{\mathrm{Bor}}
\newcommand{\Aut}{\mathrm{Aut}}
\newcommand{\bD}{\mathbf{D}}
\newcommand{\cT}{\mathcal{T}}
\newcommand{\bT}{\mathbf{T}}
\newcommand{\bG}{\mathbf{G}}
\newcommand{\cL}{\mathcal{L}}
\newcommand{\cI}{\mathcal{I}}

\title{Tilting modules for classical Lie superalgebras} 
\author[Chen]{Chih-Whi Chen}
\address{School of Mathematical Sciences, Xiamen University, Xiamen, China} \email{chihwhichen@xmu.edu.cn}
\author[Cheng]{Shun-Jen Cheng}
\address{Institute of Mathematics, Academia Sinica, Taipei,
Taiwan 10617} \email{chengsj@math.sinica.edu.tw}
\author[Coulembier]{Kevin Coulembier}
\address{School of Mathematics and Statistics, University of Sydney, Australia} \email{kevin.coulembier@sydney.edu.au}

\begin{document}
	\begin{abstract} We study tilting and projective-injective modules in a parabolic BGG category $\cO$ for an arbitrary classical Lie superalgebra. We establish a version of Ringel duality for this type of Lie superalgebras which allows to express the characters of tilting modules in terms of those of simple modules in that category. We also obtain a classification of projective-injective modules in the full BGG category $\mc O$ for all simple classical Lie superalgebras. We then classify and give an explicit combinatorial description of parabolic subalgebras of the periplectic Lie superalgebras and apply our results to study their tilting modules in more detail.
	\end{abstract}

\numberwithin{equation}{section}

\maketitle


\section*{Introduction}

A fundamental problem in the Lie superalgebra theory is the study of their representations. In the last decade or so there has been much progress in this direction, especially in the representation theory of simple complex finite-dimensional Lie superalgebras that are of classical type. This is mainly due to the discovery of connections between the representation theory of such Lie superalgebras with other areas of classical Lie theory such as quantum groups, quantum symmetric pairs and their canonical bases etc., see, e.g., \cite{Se96, Br03, CLW11, BW18, CLW15}. While this has led to a renewed surge of interest in the representation theory of these simple classical Lie superalgebras and an intense study of their representation categories, at present, analogous categories for a more general classical Lie superalgebra has received less attention. Recall that a finite-dimensional Lie superalgebra $\fg=\fg_{\oa}\oplus\fg_{\ob}$ is called {\em classical} if $\fg_\oa$ is a reductive Lie algebra and $\fg_\ob$ is a completely reducible $\fg_\oa$-module. It follows from the classification of finite-dimensional simple complex Lie superalgebras in \cite{Kac77} that every simple Lie superalgebra is classical except for those belonging to the so-called Cartan series.

One of the main motivations for the present paper is our attempt to understand tilting modules of a general classical Lie superalgebra in a general parabolic BGG category and to study them in a systematic fashion. The notion of tilting modules comes from the representation theory of finite-dimensional algebras. In \cite{Ri}, Ringel established the so-called Ringel duality, which exhibits a symmetry in the setting of quasi-hereditary algebras, also see, e.g., \cite{CPS, DR89, Do86, Do93}. Soergel adapted in \cite{So} Ringel's argument  to the BGG category $\cO$ of Lie algebras. In particular, the category $\cO$ is Ringel self-dual. As a consequence, characters of tilting modules in a category $\cO$ over a Kac-Moody Lie algebra can be expressed in terms of those of simple highest weight modules.

Now, tilting modules of most of the simple classical Lie superalgebras have been studied in detail in \cite{Br04} following Soergel's approach. In particular, it follows that for these Lie superalgebras the computation of irreducible characters is equivalent to that of characters of tilting modules. The computation of these characters for the general linear Lie superalgebra by means of certain canonical bases of Lusztig was formulated as a conjecture in \cite{Br03}. The conjecture was then established in \cite{CLW15} (see also \cite{BLW}). We refer to \cite{Br04q} and \cite{CLW11} for treatments of some of the other simple classical Lie superalgebras. However, the technical assumptions made in \cite{Br04, So} do not allow to include the case of periplectic Lie superalgebras. It is worth pointing out that while the BGG categories $\cO$ for the other finite-dimensional simple Lie superalgebras are not completely understood yet at present, the irreducible character problem however has a satisfactory solution, with the exception of that for the periplectic Lie superalgebra.

Also, projective-injective modules for classical Lie superalgebras have been studied before, most prominently in \cite{Ma}. However, the results in {\em op.~cit.}~are only applicable to Lie superalgebras that possess a simple-preserving (up to parity change) duality. Again, this assumption is not satisfied for the periplectic Lie superalgebra, and so the results therein are not applicable here.

The goal of this paper is to study tilting and projective-injective modules in a parabolic category $\cO$ for a general (not necessarily simple) classical Lie superalgebra. Special attention is paid to the periplectic Lie superalgebra throughout. A considerable part of our efforts in this paper is therefore spent on obtaining generalisations of results in \cite{Br04, Ma} to a setting that allow to include the periplectic Lie superalgebra. Besides such results, the paper in addition includes the following two main results. For every simple classical Lie superalgebra, we give an explicit description of the highest weights of those simple modules in category $\cO$ which have an injective projective cover. We also classify and explicitly describe all parabolic subalgebras of the periplectic Lie superalgebra and describe explicitly the inclusion order on the set of so-called {\em reduced} parabolic subalgebras (see Section \ref{DefBPS} for precise definition).

The paper is organised as follows.  In Section~\ref{SecPrel}, we provide some background materials on classical Lie superalgebras. In particular, we review the representation categories, parabolic decompositions and Borel subalgebras. We give in Section~\ref{SecClass} a description of classical Lie superalgebras, along with some other general technical results that are to be used in the sequel.

In Section~\ref{SecO} we set up the usual description of parabolic category $\cO$ as a highest weight category. We establish a version of Ringel duality, which allows to express the character formulae of tilting modules in terms of those of simple (or equivalently) projective modules. In Section~\ref{SecPI}, we present various classification results of projective-injective modules in parabolic category $\cO$. We obtain Irving-type description that makes explicit relationship between projective-injective modules and tilting modules and socles of Verma modules, generalising results in \cite{Ma}. Also, for the full category $\cO$, and for every simple classical Lie superalgebra, we determine explicitly the projective-injective modules.

Finally, in Section~\ref{Sect::newparaboDec} we give explicit combinatorial description of parabolic subalgebras of the periplectic Lie superalgebras and make concrete some of our general results for these Lie superalgebras. Also, we describe the projective-injective modules and tilting modules with respect to arbitrary Borel subalgebras.

{\bf Acknowledgments.}
The second author is partially supported by a MoST grant and the third author by ARC grant DE170100623. We are grateful to Walter Mazorchuk and Weiqiang Wang for numerous interesting discussions.

\section{Preliminaries}\label{SecPrel}
Throughout the paper the symbols  $\mC$, $\mR$, $\mZ$, and $\mN$ stand for the sets of all complex numbers, real numbers, integers and non-negative integers, respectively. Denote the abelian group of two elements by $\mZ_2 =\{\oa,\ob\}$.  All vector spaces, algebras, tensor products, et
cetera, are over $\mC$.

\subsection{Highest weight categories} We recall some definitions from \cite{BS18}. We say that a $\mC$-linear category $\cA$ is {\bf schurian} if it is abelian, all objects have finite length, all morphism spaces are finite dimensional and $\cA$ has enough injective and projective objects.

Let $\cA$ be a schurian category. We label the set of isomorphism classes of simple objects of $\cA$ {by $\{L(\la)|~\la \in \Lambda\}$, where $\Lambda$ is the index set.}
For $\lambda\in \Lambda$, we denote the projective cover and injective hull of the simple module $L(\lambda)$ by $P(\lambda)$ and $I(\lambda)$, respectively. For a partial order $\le$ on $\Lambda$ the standard object $\Delta(\lambda)$ is the maximal quotient of $P(\lambda)$ for which each composition factor is labelled by some $\mu\le \lambda$.

\begin{df}
For a schurian category $\cA$ and a partial order $\le$ on $\Lambda$, the pair $(\cA,\le)$ is a {\bf highest weight category} if we have
\begin{enumerate}
\item $[\Delta(\lambda):L(\lambda)]=1$;
\item for each $\mu\in\Lambda$, the object $P(\mu)$ has a filtration with each quotient of the form $\Delta(\nu)$ for some $\nu\ge\mu$.
\end{enumerate}
\end{df}

The costandard object $\nabla(\lambda)$ is the maximal subobject of $I(\lambda)$ for which each composition factor is labelled by some $\mu\le \lambda$.
We denote by $\cF(\Delta)$, respectively $\cF(\nabla)$, the full subcategory of $\cA$ of objects which admit filtrations with each quotient of the form $\Delta(\mu)$, respectively $\nabla(\mu)$. Assume that $(\cA,\le)$ is a highest weight category, so in particular projective objects are in $\cF(\Delta)$.
By definition, see e.g.~\cite{BS18, Hu08, Ri}, the tilting objects in $(\cA,\le)$ are the objects in $\cF(\Delta)\cap\cF(\nabla)$.
The following standard results about highest weight categories, see e.g.~\cite{BS18, CPS, Hu08, Ri}, will be freely used.
\begin{prop}
Let $(\cA,\le)$ be a highest weight category.
\begin{enumerate}
\item For every $\lambda\in \Lambda$, we have $I(\lambda)\in \cF(\nabla)$.
\item A direct summand of a tilting object is also a tilting object.
\end{enumerate}
\end{prop}
\begin{proof}
Claim (i) is proved in \cite[Theorem 3.6]{BS18}, claim (ii) is proved in \cite[Lemma~4.1]{BS18}.
\end{proof}

\subsection{Functors between representation categories of Lie superalgebras}
Fix a finite dimensional Lie superalgebra $\fa=\fa_{\oa}\oplus\fa_{\ob}$, see, e.g., \cite{Kac77}. We denote by $\cC(\fa)$ the category of $\mZ_2$-graded $U(\fa)$-modules, with parity preserving module morphisms. The parity shift functor on $\cC(\fa)$ is denoted by $\Pi$. We will consider the category of $\mZ_2$-graded vector spaces as the symmetric monoidal category of super vector spaces. This means that the braiding isomorphism is given by $v\otimes w\mapsto (-1)^{|v||w|}w\otimes v$, for $v,w$ homogenous vectors. Since the universal enveloping algebra $U(\fa)$ is a cocommutative Hopf algebra in the category of super vector spaces, $\cC(\fa)$ is a symmetric monoidal category.
Observe that in case $\fa$ is actually a Lie algebra, i.e., $\fa_{\ob}=0$, the category $\cC(\fa)$ is the direct sum of two copies of the usual representation category. Unless stated otherwise, an ordinary $\fa$-module interpreted as in $\cC(\fa)$ is then assumed to be purely even.

For a $U(\fa)$-module $V$ concentrated in odd degree, its symmetric algebra $SV$ is a finite dimensional $U(\fa)$-module. We write $S^{top}V$ for the one-dimensional direct summand in degree $\dim V_{\ob}$, which also has the same parity as $\dim V_{\ob}$.

Let $\fb=\fb_{\oa}$ be an even subalgebra $\fb\subset\fa$ such that $\fa$ is a semisimple $\fb$-module under the adjoint action (in particular, $\fb$ is reductive).
We denote by $\cC(\fa,\fb)$ the full subcategory of $\cC(\fa)$ consisting of those modules that are semisimple and locally finite as $\fb$-modules. Furthermore, we let $\cC^f(\fa,\fb)$ denote the full subcategory of $\cC(\fa,\fb)$ consisting of modules that have finite multiplicities (as semisimple $\fb$-modules). We denote by
$$\Gamma_{\fb}:\;\cC(\fa)\to\cC(\fa,\fb),$$
the functor right adjoint to the inclusion functor. In other words, $\Gamma_{\fb}M$ is  the maximal submodule of $M$ on which $\fb$ acts semisimply and locally finitely. 

For a subalgebra $\fc\subset\fa$, we denote by $\Res^{\fa}_{\fc}$ the forgetful functor from $\cC(\fa)$ to $\cC(\fc)$. This functor has left and right adjoint functors
$$\Ind^{\fa}_{\fc}=U(\fa)\otimes_{U(\fc)}-\qquad\mbox{and}\qquad \Coind^{\fa}_{\fc}=\Hom_{U(\fc)}(U(\fa),-).$$
If $\fc\subset\fa$ contains $\fa_{\oa}$, then \cite[Theorem~2.2]{BF} implies that
\begin{equation}\label{eqBF}\Ind^{\fa}_{\fc}\;\cong\; \Coind^{\fa}_{\fc}(S^{top}(\fa/\fc)\otimes -).\end{equation}
In other words, the induction and coinduction functors are isomorphic up to taking the tensor product with the one-dimensional $\fc$-module realised as the top symmetric power of the purely odd superspace $\fa/\fc$.

When it is clear which superalgebra $\fa$ is considered, the undecorated notations $\Res$, $\Ind$, $\Coind$ will refer to the functors acting between $\cC(\fa)$ and $\cC(\fa_{\oa})$.

\subsection{Classical Lie superalgebras} \label{Sect::claLiesup}
A finite-dimensional Lie superalgebra $\fg=\fg_{\oa}\oplus\fg_{\ob}$ is called {\bf classical} if the restriction of the adjoint representation of $\fg$ to the Lie algebra $\fg_{\oa}$ is completely reducible. In particular, the even subalgebra $\fg_{\oa}$ is a reductive Lie algebra. Now fix a Cartan subalgebra $\fh_{\oa}$ of $\fg_{\oa}$. We then have a set of roots $\Phi\subset\fh_{\oa}^\ast$ and a decomposition of vector spaces
$$\fg\;=\;\bigoplus_{\alpha\in\Phi\amalg \{0\}}\fg^\alpha,\qquad\mbox{with }\;\fg^\alpha=\{X\in\fg\,|\, [H,X]=\alpha(H)X,\,\mbox{ for all $H\in\fh_{\oa}$}\}.$$
We define $\fh:=\fg^0$, and refer to it as the {\bf Cartan subalgebra} of $\fg$.
For any $\fh_{\oa}$-submodule $V$ of $\fg$, we write $\Phi(V)$ for the subset of $\Phi$ of weights appearing in $V$.
In particular, we write $\Phi_{\oa}=\Phi(\fg_{\oa})$ and $\Phi_{\ob}=\Phi(\fg_{\ob})$.

The Weyl group $W$ of $\fg$ is by definition the Weyl group $W(\fg_{\oa}:\fh_{\oa})$. If the choice of Borel subalgebra $\fb_{\oa}\subset\fg_{\oa}$ is clear, the length function on $W$ is denoted by $\ell:W\to\mN$. We denote by $w_0$ the longest element of $W$. Since each automorphism $\exp(\ad X)$ of $\fg_{\oa}$, for $X$ nilpotent in $\fg_{\oa}$, defines an automorphism on $\fg_{\ob}$,
the action of $w\in W$ on $\fg_{\oa}$ extends to an automorphism $\varphi^{w}\in\Aut(\fg)$.

We introduce a duality $\bD$ on $\cC^f(\fg,\fh_{\oa})$, which twists the canonical duality by the automorphism $\varphi^{w_0}$. For $M\in\cC^f(\fg,\fh_{\oa})$, we set
$$\bD M\;=\; (\Gamma_{\fh_{\oa}}M^\ast)_{\varphi^{w_0}}\;\cong\; \Gamma_{\fh_{\oa}}(M^\ast_{\varphi^{w_0}}).$$
Here, for any $N\in \cC(\fg)$, the notation $N^\ast$ stands for the superspace of linear functionals $f:N\to \mC$ with action of $X\in\fg$ given by $X(f)(n)=-(-1)^{|X||f|}f(Xn)$, $n\in N$.

We will consider the abelian group $\fh^\ast\times\mZ_2$ with {trivial action of the Weyl group on $\mZ_2$}. For $M\in\cC(\fg,\fh_{\oa})$ and $\kappa=(\lambda,i)\in \fh^\ast\times\mZ_2$ we have the homogeneous weight spaces
$$M_\kappa\,:=\,\{v\in M_i\,|\, Hv=\lambda(H)v,\;\mbox{ for all $H\in\fh_{\oa}$}\}.$$

A {\em character} is a function $\fh_\oa^\ast\times \mZ_2\to\mN$. For $M\in \cC^f(\fg,\fh_{\oa})$, we have
$$\ch M:\fh^\ast_{\oa}\times\mZ_2\to\mN,\qquad \kappa\mapsto\dim M_\kappa.$$
We will usually express characters as (infinite) sums of the basis elements $e^\nu$ for $\nu\in\fh^\ast\times\mZ_2$ which satisfy $e^\nu(\kappa)=\delta_{\mu\kappa}.$ Moreover, $e^\nu e^\mu=e^{\nu+\mu}$.

For an $\fh_{\oa}$-submodule $\fk$ of the adjoint representation of $\fg$, we define $\rho(\fk)\in\fh^\ast_{\oa}$ as
$$\rho(\fk)\;=\;\frac{1}{2}\sum_{\alpha\in\Phi} (\dim \fk^\alpha_{\oa}-\dim\fk^\alpha_{\ob})\alpha.$$
In the above equation $\fk^\alpha$ stands for $\fk\cap\fg^\alpha$. For example, we always have $\rho(\fg_{\oa})=0$. However, even for simple classical Lie superalgebras, we can have $\rho(\fg)\not=0$. The latter is namely the case for periplectic superalgebras. To simplify notation, we will also write $2\rho(\fk)$ for the element $(2\rho(\fk),i)\in \fh^\ast_{\oa}\times \mZ/2$ with $i$ the parity of $\dim \fk_{\ob}$ when appropriate.

\subsection{Parabolic decompositions} \label{DefBPS}
We continue to let $\fg$ be a classical Lie superalgebra with fixed Cartan subalgebra $\fh_{\oa}$ of $\fg_{\oa}$.
We follow the notion of parabolic decompositions of superalgebras from \cite[\S 2.4]{Ma}. For each $H\in\fh_{\oa}$ we can define subalgebras of $\fg$
\begin{equation}\label{deflu}\fl:=\bigoplus_{\Real \alpha(H)=0} \fg^\alpha,\quad \fu^+:=\bigoplus_{\Real \alpha(H)>0} \fg^\alpha, \quad \fu^-:=\bigoplus_{\Real \alpha(H)<0} \fg^\alpha,\quad\mbox{with}\quad\fg=\fu^-\oplus\fl\oplus\fu^+,\end{equation} where $\Real(z)$ denotes the real part of $z\in \mathbb C$.
We write $\fl(H)$ and $\fu^{\pm}(H)$ for the above algebras when it is necessary to keep track of $H\in\fh_{\oa}$.

Decompositions into such subalgebras of $\fg$ as above are the {\bf parabolic decompositions of $\fg$}. We define {\bf a Levi subalgebra} of $\fg$ to be a subalgebra $\fl$ as above, for some $H\in\fh_{\oa}$.
Similarly, we define {\bf a parabolic subalgebra} of $\fg$ to be a subalgebra $\fp$ which is of the form $\fl\oplus\fu^+$ as above. We set $\fp^-=\fp(-H)=\fl\oplus\fu^-$.
A parabolic decomposition of $\fg$ is determined by the corresponding pair $(\fp,\fl)$ of subalgebras (since $\fh_{\oa}$ is fixed). However, different parabolic decompositions can lead to the same parabolic subalgebra, see Example~\ref{examp2}.

A given parabolic subalgebra $\fp$ contains at most one Levi subalgebra $\fl$ which satisfies $\fl=\fl_{\oa}$ (again since $\fh_{\oa}$ is fixed). If such a Levi subalgebra exists we say that $\fp$ is a {\bf reduced parabolic subalgebra}. With slight abuse of terminology we will often refer to the purely even Levi subalgebra of a reduced parabolic subalgebra $\fp$ simply as `the' Levi subalgebra of $\fp$. Note that reduced parabolic subalgebras only exist when $\fh=\fh_{\oa}$. For a given parabolic subalgebra $\fp_{\oa}\subset\fg_{\oa}$, we denote by
$\Par(\fg,\fp_{\oa})$
the set of reduced parabolic subalgebras of $\fg$ which have $\fp_{\oa}$ as underlying even subalgebra.

We say that $H\in\fh_{\oa}$ is regular if and only if $\fl(H)=\fg^0=\fh$. In this case, we usually write $\fn^{\pm}=\fu^{\pm}$. Such a decomposition $\fn^-\oplus\fh\oplus\fn^+$ gives rise to a {triangular decomposition of $\fg$, i.e., a decomposition into subalgebras with $[\fh,\fn^\pm]\subseteq\fn^\pm$}. Following \cite{PS}, we define the {\bf Borel subalgebras} of $\fg$ to be the subalgebras $\fh\oplus\fn^+$ obtained as above from regular $H\in\fh_{\oa}$. {In \cite[Section 3.2]{Mu12}, they are referred to as BPS subalgebras.}
For a Borel subalgebra $\fb=\fh\oplus\fn^+$, the subalgebra $\fb^-:=\fh\oplus\fn^-$ is also a Borel subalgebra. For a Borel subalgebra $\fb_{\oa}$ of $\fg_{\oa}$, we denote by $\Bor(\fg,\fb_{\oa})$ the set of Borel subalgebras of $\fg$ which have $\fb_{\oa}$ as underlying Lie algebra. When it is clear which Borel subalgebra is considered we will simply write $\rho$ for
$$\rho:=\rho(\fb)=\rho(\fn^+)=\sum_{\alpha\in\Phi}(\dim\fb^\alpha_{\oa}-\dim\fb^\alpha_{\ob}){\alpha}.$$

\begin{lem}\label{Lembp}
For a parabolic subalgebra $\fp_{\oa}\subset\fg_{\oa}$ containing a Borel subalgebra $\fb_{\oa}$, we have an injection
$$\Par(\fg,\fp_{\oa})\,\hookrightarrow\; \Bor(\fg,\fb_{\oa}),\quad \fp\mapsto \fb_{\oa}\oplus\fp_{\ob}.$$
\end{lem}
\begin{proof}
We may assume that $\fh=\fh_{\oa}$ since otherwise $\Par(\fg,\fp_{\oa})=\emptyset$.

We need to show that $\fb_{\oa}\oplus\fp_{\ob}$ is a Borel subalgebra of $\fg$, for $\fp\in\Par(\fg,\fp_{\oa})$.

Assume that $\fp=\fp(H)$ for some $H\in\fh$.
We shall first prove that there exists a regular $H'\in\fh$ with $\fb(H')\subset \fp(H)$. We define the following subsets of $\fh$:
\begin{align*}
C\;&=\;\{X\in\fh\,|\, \Real\alpha(X)\Real\alpha(H)>0,\;\forall \alpha\in{\Phi(\fu^+)\cup\Phi(\fu^-)}\},\\
D\;&=\;\{X\in\fh\,|\, \Real\alpha(X)\not=0,\;\forall \alpha\in{\Phi}\}.
\end{align*}
We claim that $C\cap D\not=\emptyset$. Choosing an $\mR$-basis of $\fh$ allows us to identify $\fh\cong\mR^{2r}$, for $r=\dim_{\mC}\fh$.
As a finite intersection of open half-spaces, $C$ is open in the Euclidean topology. Since $H\in C$, it follows that $C$ contains an open ball $B$ with centre $H$.
Furthermore, $D$ is the complement of the union of a finite set of hyperplanes. Therefore, $B$ must intersect $D$ non-trivially, proving the claim. Since any $H'\in C\cap D$ is regular by construction, $\fb(H')$ is a Borel subalgebra. Furthermore, we have $\fb(H')\subset \fp(H)$ since $H'\in C$.

Since $\fb(H')_{\oa}$ is a Borel subalgebra inside $\fp(H)_{\oa}=\fp_{\oa}$, it is conjugate to $\fb_{\oa}$ via an automorphism of $\fl=\fl_{\oa}$. Since such an automorphism leaves $\fu^+$ invariant, we see that $\fb_{\oa}\oplus\fp_{\ob}$ is indeed a Borel subalgebra.

As the map is clearly injective, this completes the proof.
\end{proof}

Fix a Borel subalgebra $\fb_{\oa}$ of $\fg_{\oa}$. By the extension of the action of $W$ to $\fg$ it follows that every Borel subalgebra of $\fg$ is conjugate to one that has $\fb_{\oa}$ as underlying even subalgebra. More generally, up to conjugation we can assume that each parabolic subalgebra contains $\fb_{\oa}$.

The action $\varphi^{w_0}$ of $w_0\in W$ on $\fg$ also allows us to define a duality on the set of parabolic decompositions. For simplicity we restrict to reduced parabolic subalgebras.
For a reduced parabolic subalgebra $\fp$, the parabolic subalgebra $\hat{\fp}$ is defined as $\hat{\fp}=\varphi^{w_0}(\fp)^-$. That this is indeed a parabolic subalgebra follows from the direct observation that if $\fp=\fp(H)$, for $H\in\fh$, then $\hat{\fp}=\fp(-w_0(H))$.
We will use the same notation for this duality for parabolic subalgebras of $\fg_{\oa}$. Note that $\hat{\fb}_{\oa}=\fb_{\oa}$ but in general $\hat{\fp}_{\oa}\not=\fp_{\oa}$ and $\hat{\fl}\not=\fl$. In general we have
$$\alpha\in\Phi(\fu^{\pm})\qquad\mbox{if and only if}\qquad w_0\alpha\in\Phi(\hat{\fu}^{\mp}).$$

\section{Structure of classical Lie superalgebras}\label{SecClass}

In this section we aim to explore how rich the family of classical Lie superalgebras is, compared to the well-studied family of {\em simple} classical Lie superalgebras, classified in \cite{SNR} and \cite{Kac77}. We explain how all the semisimple classical Lie superalgebras can be obtained from the simple ones and subsequently how all the classical Lie superalgebras can be obtained from the semisimple ones. The latter result is equivalent to \cite[Theorem~B]{El}. We include a proof, as the result is essentially a byproduct of the proof of another result that we will need later in the paper.

\begin{df}
For a Lie superalgebra $\fg$, the radical $\rad(\fg)$ is the sum of all solvable ideals of $\fg$, which is thus the unique maximal solvable ideal of $\fg$. A Lie superalgebra $\fg$ is semisimple if $\rad(\fg)=0$.\end{df}

\subsection{Classical semisimple Lie superalgebras}

Finite-dimensional semisimple Lie algebras in prime characteristic are described by means of their differentiably simple ideals in \cite[Theorem 9.3]{Block69}, which reduces the problem of classification of semisimple Lie algebras to the problem of determining the so-called differentiably simple Lie algebras in prime characteristic. This problem was solved in \cite[Main Theorem]{Block69}. It was suggested in \cite {Kac77} that the methods in \cite{Block69} can be suitably modified to give a classification of differentiably simple Lie superalgebras in characteristic zero. Applying the same approach as in \cite{Block69} then gives a classification of semisimple Lie superalgebras in characteristic zero. The details were carried out in \cite[Theorem 6.1]{Ch95} and \cite[Proposition 7.2]{Ch95}. We shall refrain ourselves from going into more details, but rather restrict ourselves to the case that we are interested in, i.e., we shall state the classification of the classical semisimple Lie superalgebras obtained by applying \cite[Theorem 6.1]{Ch95} and \cite[Proposition 7.2]{Ch95}. In order to state this we shall need some preparation.

Let $\wedge(\xi)$ be the Grassmann superalgebra in the odd indeterminate $\xi$. That is, $\wedge(\xi)$ is the associative superalgebra spanned over $\C$ by the even identity vector $1$ and odd vector $\xi$ with $\xi^2=0$. Let $S$ be a finite-dimensional simple Lie algebras and consider
\begin{align*}
S_{[1]}:=S\otimes\wedge(\xi)
\end{align*}
which is a Lie superalgebra with obvious Lie bracket. The algebra of derivations $W_{[1]}=W(\xi)$ of the associative superalgebra $\wedge(\xi)$ is the Lie superalgebra spanned by the odd vector $\partial\over{\partial \xi}$ and the even vector $\xi{\partial\over{\partial \xi}}$, determined by
\begin{align*}
{\partial\over{\partial \xi}}(1)=\xi{\partial\over{\partial \xi}}(1)=0,\quad
{\partial\over{\partial \xi}}(\xi)=1,\quad \xi{\partial\over{\partial \xi}}(\xi)=\xi.
\end{align*}
The Lie superalgebra of derivations $W_{[1]}$ acts on $S_{[1]}:=S\otimes\wedge(\xi)$ in a natural way so that we can form the semidirect sum $S_{[1]}\rtimes W_{[1]}$.

For a finite-dimensional simple Lie superalgebra $L$ we let $\texttt{der}L$ denote its Lie superalgebra of derivations. Recall that $\texttt{der}L$ contains $L$ as the ideal of inner derivations. The Lie superalgebras of outer derivations $\mathfrak d_L:=\texttt{der}L/L$ of all simple Lie superalgebras are described in \cite[Proposition 5.1.2 ]{Kac77} and for the classical simple Lie superalgebras we have
\begin{align*}
\texttt{der}L=L\rtimes\mathfrak d_L.
\end{align*}

We are now ready to state the following theorem, which is a direct consequence of \cite[Proposition 7.2]{Ch95}.

\begin{thm}\label{thm:str:ss} Let $m,n\in\N$ with $m+n>0$. Let $S^1,\ldots,S^m$ be finite-dimensional simple Lie algebras and let $L^1,\ldots,L^n$ be classical simple Lie superalgebras. (Here $L^j$ is allowed to be a Lie algebra.) We consider a Lie superalgebra $\fg$ with the properties that
\begin{align*}
\bigoplus_{i=1}^m S^i_{[1]}\oplus\bigoplus_{j=1}^n L^j\subseteq \fg \subseteq \bigoplus_{i=1}^m \left(S^i_{[1]}\rtimes W^i_{[1]}\right)\oplus\bigoplus_{j=1}^n \left(L^j\rtimes\mathfrak d_{L^j}\right),
\end{align*}
and that the projection of $\fg$ to each subspace $\C {\partial\over{\partial \xi_{i}}}$ of $W^i_{[1]}$ is surjective. Then we have:
\begin{itemize}
\item[(i)] The radical of $\fg$ is trivial and hence $\fg$ is a classical semisimple Lie superalgebra.
\item[(ii)] Any classical semisimple Lie superalgebra is isomorphic to such a $\fg$ above.
\end{itemize}
\end{thm}


\begin{ex}
Suppose that $\fg$ is a Lie algebra in Theorem \ref{thm:str:ss}. Then in this case we must have $m=0$ and every $L^j$ must be a simple Lie algebra, so that $\texttt{der}L^j=L^j$, for all $j$. It follows that $\g=\bigoplus_{j=1}^n L^j$, and we get the well-known classification of semisimple Lie algebras as a direct sum of simple Lie algebras.
\end{ex}

\begin{ex}\label{classical:semi1}
Suppose that $m=1$ and $n=0$ in Theorem \ref{thm:str:ss}. In this case we get two examples of classical semisimple Lie superalgebras, namely:
$$
S\otimes\wedge(\xi)\rtimes \C {\partial\over{\partial \xi}}+\C \xi{\partial\over{\partial \xi}},\quad\mbox{and}\quad S\otimes\wedge(\xi)\rtimes \C {\partial\over{\partial \xi}}.$$
\end{ex}

\begin{ex} Suppose that $m=0$ and $n=1$ in Theorem \ref{thm:str:ss}. Then the Lie superalgebra $\fg$ with $$L\subseteq \fg \subseteq L\rtimes\mathfrak d_L$$ has trivial radical and so is a semisimple Lie superalgebra. An example with $L\subsetneq \fg$ is the Lie superalgebra $\mathfrak{pgl}(k|k):=\mathfrak{gl}(k|k)/\C I$ with $k\ge 2$ and $I$ is the identity element.
\end{ex}

\subsection{A classification in terms of semisimple Lie superalgebras}
For a classical Lie superalgebra $\fg$, a representation contained in one degree clearly factors through a representation of the reductive Lie algebra $\fg_{\oa}/[\fg_{\ob},\fg_{\ob}]$. For a purely odd, semisimple, finite-dimensional representation $V$ of $\fg$ (or of $\fg_{\oa}/[\fg_{\ob},\fg_{\ob}]$), we interpret $V$ as a purely odd abelian Lie superalgebra, which allows to define the semidirect sum $\fg \ltimes V$, which is again classical.

The following is a reformulation of \cite[Theorem~B]{El}.
\begin{thm}\label{ThmClass}
Every classical Lie superalgebra is an even central extension of a Lie superalgebra of the form $\fg\ltimes V$, with $\fg$ an even central extension of a classical semisimple Lie superalgebra and $V$ a
purely odd, semisimple, finite-dimensional $\fg$-module.
\end{thm}
\begin{proof}
We start from an arbitrary classical Lie superalgebra $\fa$, with centre $\fz=\fz(\fa)$. Hence $\fa$ is an even central extension of $\fb:=\fa/\fz_{\oa}$ and $\fz(\fb)_{\oa}=0$. By Theorem~\ref{ThmRad} below, we have
$$\rad(\fb)_{\ob}\;=\;\{X\in\fb_{\ob}\,|\, [X,\fb_{\ob}]=0\}.$$
In particular, $\rad(\fb)_{\ob}$ is an ideal in $\fb$ and we set $\fc:=\fb/\rad(\fb)_{\ob}$. Since we can take a complement of $\rad(\fb)_{\ob}$ in the semisimple $\fb_{\oa}$-module $\fb_{\ob}$, we find $\fb\cong\fc\ltimes \rad(\fb)_{\ob}.$ By construction, we have $\rad(\fc)_{\ob}=0$ and it follows that
$$\rad(\fc)\;=\;\fz(\fc)\;=\;\fz(\fc)_{\oa}.$$
This concludes the proof.
\end{proof}

\begin{ex}\label{ExSDP}
\begin{enumerate}
\item The Lie superalgebra $\mathfrak{gl}(1|1)$ is a central extension of $\mathfrak{pgl}(1|1)$, where the latter is of the form $\mC\ltimes \Pi\mC^2$, with $\mC$ the one-dimensional (reductive) abelian Lie algebra and $\mC^2$ equipped with the structure of a non-trivial self-dual representation of $\mC$.
\item The classical Lie superalgebras $\fg$ of the form $\fg_{\oa}\ltimes \Pi V$ for $V$ an arbitrary finite-dimensional semisimple representation of $\fg_{\oa}$ are known as the generalised Takiff superalgebras, see, e.g., \cite{Ma}.
\end{enumerate}
\end{ex}

\begin{rem}
\begin{enumerate}
\item The converse to Theorem~\ref{ThmClass} is not true. Concretely, an even central extension of a classical Lie superalgebra need not be classical. Already central extensions of reductive Lie algebras need not be reductive, see for example Heisenberg Lie algebras.
\item The procedure in Theorem~\ref{ThmClass} is minimal in the sense that in general both central extensions are required. Examples of such algebras are given by $\mathfrak{gl}(1|1)$, see Example~\ref{ExSDP}(1), and the following construction. Let $S$ be a simple Lie algebra so that $\fg=S\otimes\wedge(\xi)+\mC\xi{\partial\over{\partial\xi}}$ is a Takiff superalgebra. Let $\widehat{\fg}$ be the even central extension of $\fg$ by a central element $z$ such that the only commutator of $\widehat{\fg}$ different from the one of $\fg$ is $[s\otimes \xi,s'\otimes\xi]=(s,s')z$, where $s,s'\in S$ and $(\cdot,\cdot)$ is the Killing form on $S$.
\end{enumerate}
\end{rem}

\subsection{The radical of a classical Lie superalgebra}For this section, we fix an arbitrary classical Lie superalgebra $\fg$.

\begin{thm}\label{ThmRad}
The odd part of the radical of $\fg$ is given by
$$\rad(\fg)_{\ob}\;=\;\{X\in\fg_{\ob}\,|\, [X,\fg_{\ob}]\subset\mathfrak{z}(\fg)_{\oa}\}.$$
\end{thm}

The proof will be a direct consequence of the following lemma, which also contains some additional results which will be useful later.  By assumption $\fg_{\oa}$ is reductive, so we have a canonical sum $\fg_{\oa}=\fs\oplus\fa$, where $\fs$ is a semisimple Lie algebra and $\fa=\mathfrak{z}(\fg_{\oa})$ is abelian.

\begin{lem}\label{Lemabc}
We have a vector space decomposition
$$[\fg_{\ob},\fg_{\ob}]\;=\;[\fg_{\ob},\fg_{\ob}]\cap \fs\,\oplus [\fg_{\ob},\fg_{\ob}]\cap\fa.$$
For each $H\in [\fg_{\ob},\fg_{\ob}]\cap\fa$, the trace of the action $\ad_H$ on $\fg_{\ob}$ is zero, and precisely one of the following is true:
\begin{enumerate}
\item $H$ is in the centre $\mathfrak{z}(\fg)$ of $\fg$.
\item $H$ is not in the radical $\rad(\fg)$ of $\fg$.
\end{enumerate}
\end{lem}
\begin{proof}
The vector space decomposition follows immediately from the fact that $[\fg_{\ob},\fg_{\ob}]$ is a sub-representation of the (semisimple) adjoint representation of $\fg_{\oa}$.

Choose a (non-canonical) decomposition $\fg_{\ob}=\oplus_{\alpha\in E} V_\alpha$ of $\fg_{\ob}$ into irreducible $\fg_{\oa}$-representations for the adjoint action.
From the Jacobi identity we get
$$[[V_\alpha,V_\beta],V_\gamma]\;\subset \;V_\gamma\cap (V_\alpha+V_\beta),\quad\mbox{for all $\alpha,\beta,\gamma\in E$.}$$
In particular, this implies
\begin{equation}\label{abc}[[V_\alpha,V_\beta],V_\gamma]=0,\quad\mbox{if both $\gamma\not=\alpha$ and $\gamma\not=\beta$}.\end{equation}

We can interpret the Lie bracket as a $\fg_{\oa}$-module morphism
$$\fg_{\ob}\otimes\fg_{\ob}\;\stackrel{[\cdot,\cdot]}{\to}\;\fg_{\oa}.$$
When restricted to $V_\alpha\otimes V_\beta$, the image is either contained in $\fs$ or has a one-dimensional intersection with $\fa$. In the latter case, $V_\alpha$ and $V_\beta$ must be dual $\fg_{\oa}$-modules.
It follows also that we have
$$[\fg_{\ob},\fg_{\ob}]\cap\fa\;=\;\sum_{\alpha,\beta\in E}([V_\alpha,V_\beta]\cap\fa).$$
For each pair $(\alpha,\beta)\in E\times E$ for which $[V_\alpha,V_\beta]\cap\fa\not=0$,
we fix a non-zero element $H_{\alpha\beta}\in[V_\alpha,V_\beta]\cap\fa$, uniquely defined up to scalar. These elements span $[\fg_{\ob},\fg_{\ob}]\cap\fa$, hence it is sufficient to study the trace of $\ad_{H_{\alpha\beta}}$.

For every relevant $\alpha\in E$, the element $H_{\alpha\alpha}$ acts as zero on each space $V_\gamma$ with $\gamma\not=\alpha$ by \eqref{abc}. Moreover, since in this case $V_\alpha$ is a self-dual $\fg_{\oa}$-representation, $H_{\alpha\alpha}$ (as an element of the centre $\fa$ of $\fg_{\oa}$) must also act trivially on $V_\alpha$. In conclusion, $H_{\alpha\alpha}$ is central in $\fg$.

Now consider an element $H_{\alpha\beta}$ for $\alpha\not=\beta$. By \eqref{abc}, the action of $H_{\alpha\beta}$ is zero on each space $V_\gamma$ for $\gamma\not\in\{\alpha,\beta\}$. Since the $\fg_{\oa}$-modules $V_\alpha$ and $V_\beta$ are dual to each other, they have the same dimension and the two eigenvalues of $H_{\alpha\beta}$ on them add up to zero. In particular, the trace of $\ad_{H_{\alpha\beta}}$ is zero.

It remains to prove the dichotomy in the lemma. In order to arrive at a contradiction, we assume that some linear combination $H$ of the $\{H_{\alpha\beta}\}$ is not central, but is in $\rad(\fg)$. Without loss of generality, we may assume that all $H_{\alpha\beta}$ which appear with non-zero coefficient in $H$ are not central. Since $H$ is not central, there will be some $H_{\alpha\beta}$ in the linear combination such that $H$ acts via a non-zero scalar on $V_\alpha$ or $V_\beta$. By symmetry we can assume that the action on $V_\alpha$ is not zero. In particular the ideal generated by $H$ contains $V_\alpha$, which means it also contains $H_{\alpha\beta}\in [V_\alpha,V_\beta]$. To conclude the proof it thus suffices to show that the ideal $I$ generated by a non-central $H_{\alpha\beta}$ is not solvable, which contradicts the assumption that $H$ (and hence $H_{\alpha\beta}$) be in the radical. Since $H_{\alpha,\beta}$ is not central, $I$ contains the subspace
$\mC H_{\alpha\beta}\oplus V_\alpha\oplus V_\beta.$ Since $H_{\alpha\beta}\in[V_\alpha,V_\beta]$, we find that $I=[I,I]$, so in particular $I$ is indeed not solvable. \end{proof}

\begin{proof}[Proof of Theorem \ref{ThmRad}]
We have obvious inclusions
$$\{X\in\fg_{\ob}\,|\, [X,\fg_{\ob}]\subset\mathfrak{z}(\fg)_{\oa}\}\,\subseteq\, \rad(\fg)_{\ob}\,\subseteq\, \{X\in\fg_{\ob}\,|\, [X,\fg_{\ob}]\subset\fa\}. $$
So for $X\in \rad(\fg)_{\ob}$ not contained in the left-hand side, we can take $Y\in \fg_{\ob}$ such that we get an element $Z:=[X,Y]\in [\fg_{\ob},\fg_{\ob}]\cap\fa$ which is not in the centre of $\fg$. By Lemma~\ref{Lemabc}, $Z$ is not in $\rad(\fg)$ which of course contradicts the assumption $X\in \rad(\fg)$. Consequently, the left-hand inclusion is an equality.
\end{proof}

 \begin{cor}\label{CorEta}
 	The one-dimensional $\fg_{\ob}$-module $S^{top}(\fg_{\ob})$ is the restriction of a $\fg$-module.
 \end{cor}
\begin{proof}
Any $\fg_{\oa}$-module $M$ is canonically the restriction of a $\fg$-module if the action of $[\fg_{\ob},\fg_{\ob}]\subset\fg_{\oa}$ on $M$ is trivial. If $M$ is one-dimensional, it suffices to check that the action of $[\fg_{\ob},\fg_{\ob}]\cap\fa$ is trivial. For $M=S^{top}(\fg_{\ob})$ this follows from Lemma~\ref{Lemabc}.
\end{proof}

\section{Parabolic category $\cO$}\label{SecO}
\subsection{Assumptions and notation} \label{Subsect::AssNota}
Consider a classical Lie superalgebra $\fg$ and a fixed triangular decomposition of the underlying Lie algebra
$$\fg_{\oa}=\fn_{\oa}^-\oplus\fh_{\oa}\oplus\fn^+_{\oa}$$
{such that $\fb_{\oa}=\fh_{\oa}\oplus\fn^+_{\oa}$ is a Borel subalgebra of $\fg_\oa$.}
We also fix a parabolic subalgebra $\fp_{\oa}\subset\fg_{\oa}$ containing $\fb_{\oa}$. We assume that $\Par(\fg,\fp_{\oa})$ is not empty. In particular, this implies $\fh=\fh_{\oa}$. By Lemma~\ref{Lembp}, to each $\fp\in \Par(\fg,\fp_{\oa})$, we can associate a unique Borel subalgebra $\fb:=\fb_{\oa}\oplus\fp_{\ob}$.

The Levi subalgebra of any $\fp$ in $\Par(\fg,\fp_{\oa})$ is by assumption the Levi subalgebra of $\fp_{\oa}$, so we denote it simply by $\fl=\fl_{\oa}$. The longest element of the Weyl group of $\fl$ will be denoted by $w_0^{\fp}$.
We say that $\lambda\in\fh^\ast$ is $\fp_{\oa}$-dominant if the simple highest weight module  $L_\fl(\lambda)$ of $\fl$ (with respect to the fixed Borel subalgebra $\fl\cap\fb_{\oa}$ of $\fl$) is finite dimensional.


\subsection{Definitions}
The category $\cO(\fg,\fp_{\oa})$ is the full subcategory of $\cC(\fg,\fh)$ of finitely generated $\mf g$-modules on which $\fp_{\oa}$ acts locally finitely. Equivalently, $\cO(\fg,\fp_{\oa})$ is the full subcategory of $\cC(\fg,\fl)$ of finitely generated modules on which $\fu^+_{\oa}$ acts locally finitely. Finally, with one of the above definitions applied to $\fg_{\oa}$, the category $\cO(\fg,\fp_{\oa})$ can be defined as the full subcategory of $\cC(\fg)$ of modules $M$ with $\Res M$ in $\cO(\fg_{\oa},\fp_{\oa})$.

Let $X_{\fp_{\oa}}$ denote the set of pairs $(\lambda,i)$ with $\lambda\in\fh^\ast$ a $\fp_{\oa}$-dominant weight and $i\in\mZ_2$.
For each $(\lambda,i)\in X_{\fp_{\oa}}$ and $\fp\in \Par(\fg,\fp_{\oa})$, we define
$$\Delta^{\fp}(\lambda,i):=\Pi^i\Ind^{\fg}_{\fp}L_{\fl}(\lambda)\quad\mbox{and}\quad \nabla^{\fp}(\lambda,i)=\Pi^i\Gamma_{\fh}\Coind^{\fg}_{\fp^-}L_{\fl}(\lambda).$$
It is a standard observation, see e.g. \cite{ChWa12, Hu08, Ma}, that $\Delta^{\fp}(\lambda,i)$ is in $\cO(\fg,\fp_{\oa})$ and has simple top. We denote the latter simple module by $L^{\fp}(\lambda,i)\equiv L(\lambda,i)$. Similarly, for $\kappa\in X_{\fp_{\oa}}$ we define the $\mf g_\oo$-modules $\Delta^{\fp_{\oa}}_\oa(\kappa)$ and $L^{\fp_{\oa}}_\oa(\kappa)$.

Furthermore, we define the partial order $\le_{\fp}$ on $\fh^\ast\times\mZ_2$ as the transitive closure of the relations
$$\begin{cases}
(\lambda-\alpha,i+j) \le_{\fp}(\lambda,i), &\mbox{for $\alpha\in \Phi(\fu^+_j)$ and $j\in\mZ_2$},\\
(\lambda+\alpha,i+j) \le_{\fp}(\lambda,i), &\mbox{for $\alpha\in \Phi(\fu^-_j)$ and $j\in\mZ_2$}.
\end{cases}$$
We use the same notation $\le_{\fp}$ for the restriction of the partial order to $X_{\fp_{\oa}}$.

\begin{thm}\label{ThmHWC}
\begin{enumerate}[(i)]
\item The category $\cO(\fg,\fp_{\oa})$ is schurian.
\item Fix an arbitrary $\fp\in\Par(\fg,\fp_{\oa})$. The assignment $\lambda\mapsto L^{\fp}(\lambda)$ yields a bijection between $X_{\fp_{\oa}}$ and the set of isoclasses of simple objects in $\cO(\fg,\fp_{\oa})$.
\item Denote by $\cO(\fg,\fp)$ the pair $(\cO(\fg,\fp_{\oa}),\le_{\fp})$ where the set of isoclasses of simple modules is identified with $X_{\fp_{\oa}}$ as in (ii). The category $\cO(\fg,\fp)$ is a highest weight category with standard objects $\Delta(\lambda)=\Delta^{\fp}(\lambda)$ and costandard objects $\nabla(\lambda)=\nabla^{\fp}(\lambda)$.
\end{enumerate}
\end{thm}
\begin{proof}
These observations are standard, see, e.g., \cite{Ma}. We shall only sketch the proof of the fact that $\cO(\fg,\fp)$ is a highest weight category below, and leave the remaining statements to the reader.

The fact that any module $M$ in $\cO(\fg,\fp_{\oa})$ has finite length follows from the observation that $\Res M$ already has finite length, see \cite[1.11]{Hu08}.
Denote by $P^{\fp_{\oa}}_\oa(\la)$ and $I^{\fp_{\oa}}_\oa(\la)$ the projective cover and injective envelop of $L_\oa^{\fp_{\oa}}(\la)$ in $\mc O(\mf g_\oo, \fp_{\oa})$, respectively. We note that the functor $\Ind$ sends projective modules to projective modules. From
\begin{align*}
0\not=\Hom_{\fg_{\oa}}(P_{\oa}^{\fp_{\oa}}(\la),\Res L^{\fp}(\la))\cong \Hom_{\fg}(\Ind P_{\oa}^{\fp_{\oa}}(\la),L^{\fp}(\la)),
\end{align*}
we see that $\cO(\fg,\fp)$ has enough projectives. Similarly, it has enough injective, and so is schurian.

The fact that $[\Delta^{\fp}(\lambda):L(\mu)]\not=0$ implies $\mu\le_{\fp}\lambda$ follows from considering $\Delta^{\fp}(\lambda)\cong U(\fu^-)\otimes L_{\fl}(\lambda)$ as an $\fl$-module.
Since $\cO(\fg_{\oa},\fp_{\oa})$ is a highest weight category, it follows that if $\Delta_{\oa}^{\fp_\oa}(\mu)$ appears in a $\Delta^{\fp_{\oa}}_\oa$-flag of $P^{\fp_{\oa}}_{\oa}(\la)$, then $\mu\ge_{\fp}\la$.
This in turn implies that any $\Delta^\fp(\nu)$ appearing in a $\Delta^\fp$-flag of $\Ind P^{\fp_{\oa}}_{\oa}(\la)$ satisfies $\nu\ge_{\fp} \la$. This property holds also for any of its direct summands, and hence $\cO(\fg,\fp)$ is a highest weight category.
\end{proof}

We denote the full subcategories of projective and injective modules in $\cO(\fg,\fp_{\oa})$ by $\cP(\fg,\fp_{\oa})$ and $\cI(\fg,\fp_{\oa})$, respectively. For given $\la\in X_{\fp_{\oa}}$, we denote the projective cover and injective envelop of $L^\fp(\la)$ in $\mc O(\mf g, \fp_{\oa})$ by $P^\fp(\la)$ and $I^\fp(\la)$, respectively.

\begin{rem}\label{RemPI}
{From the proof of Theorem~\ref{ThmHWC}, we see that objects in $\cP(\fg,\fp_{\oa})$ and $\cI(\fg,\fp_{\oa})$ are precisely the direct summands of modules of the form $\Ind Q$ and $\Coind Y$, respectively, where $Q$ is projective and $Y$ is injective in $\cO(\fg_{\oa},\fp_{\oa})$.} Note also that we can interchange induction and coinduction functor freely here, by \eqref{eqBF}.
\end{rem}

We denote by $\cF(\Delta^{\fp})$, respectively $\cF(\nabla^{\fp})$, the full subcategory of $\cC^f(\fg,\fl)$ (or equivalently of $\cO(\fg,\fp_{\oa})$) of objects which admit  finite filtrations with each quotient of the form $\Delta^{\fp}(\mu)$, respectively $\nabla^{\fp}(\mu)$. We use traditional notation $(M:\Delta^{\fp}(\lambda))$ to denote the number of times that $\Delta^{\fp}(\lambda)$ appears in such a filtration for $M\in\cF(\Delta^{\fp}) $. By the following lemma, this is well-defined.

\begin{lem}\label{LemResDelta}
\begin{enumerate}[(i)]
\item For $M\in\cF(\Delta^{\fp})$ and $N\in\cF(\nabla^{\fp})$ we have
$$(M:\Delta^{\fp}(\mu))=\dim\Hom_{\fg}(M,\nabla^{\fp}(\mu))\quad\mbox{and}\quad (N:\nabla^{\fp}(\mu))=\dim\Hom_{\fg}(\Delta^{\fp}(\mu),N).$$
\item We have
$$(P^{\fp}(\lambda):\Delta^{\fp}(\mu))=[\nabla^{\fp}(\mu):L^{\fp}(\lambda)]\quad\mbox{and}\quad (I^{\fp}(\lambda):\nabla^{\fp}(\mu))=[\Delta^{\fp}(\mu):L^{\fp}(\lambda)]$$
\item The restriction functor $\Res:\cC(\fg,\fl)\to \cC(\fg_{\oa},\fl)$ restricts to a functor $\cF(\Delta^{\fp})\to \cF(\Delta_{\oa}^{\fp_{\oa}})$.\end{enumerate}
\end{lem}
\begin{proof}
Part (i) is a generality for highest weight categories, see \cite[Theorem 3.1.4 in Section 3]{BS18}. Part (ii) is a direct application of part (i). Part (iii) follows from the fact that we can characterise $\cF(\Delta_{\oa}^{\fp_{\oa}})$ as the category of modules in $\cC(\fg_{\oa},\fl)$ that are free and finitely generated as $U(\fu_{\oa}^-)$-modules.
\end{proof}

The following proposition follows from the definition of $\bD$ and $\hat{\fp}$.
\begin{prop}\label{PropDD}
The duality $\bD$ on $\cC^f(\fg,\fh)$ restricts to a contravariant equivalence
\begin{align*}
\bD:\cO(\fg,\fp_{\oa})\to\cO(\fg,\hat{\fp}_{\oa})\text{ with } \bD\Delta^{\fp}(\lambda)\cong \nabla^{\hat\fp}(-w_0\lambda) \text{ and }
\bD\nabla^{\fp}(\lambda)\cong \Delta^{\hat\fp}(-w_0\lambda),\text{ for }\lambda\in X_{\fp_{\oa}}.
\end{align*}
\end{prop}

\subsection{Tilting modules}
Recall that the tilting modules in the highest weight category $\cO(\fg,\fp)$ are the modules in $\cF(\Delta^{\fp})\cap\cF(\nabla^{\fp})$.
We denote by $\cT(\fg,\fp_{\oa})$ the full subcategory of $\cO(\fg,\fp_{\oa})$ consisting of direct summands of modules of the form $\Ind U$, where $U$ is a tilting module in $\cO(\fg_{\oa},{\fp_{\oa}})$. The following theorem shows in particular that the question of whether a module in $\cO(\fg,\fp)$ is tilting only depends on $\fp_{\oa}$.

\begin{thm}\label{ThmTilt}
Consider an arbitrary $\fp\in\Par(\fg,\fp_{\oa})$.
\begin{enumerate}[(i)]
\item The category of tilting modules in $\cO(\fg,\fp)$ is $\cT(\fg,\fp_{\oa})$.
\item For each $\kappa \in X_{\fp_{\oa}}$, there exists an indecomposable $T^{\fp}(\kappa)\in \cT(\fg,\fp_{\oa})$, uniquely determined by the following properties:
\begin{enumerate}[(a)] \label{ThmTilt::mono}
\item We have $(T^{\fp}(\kappa):\Delta^{\fp}(\nu))=0$ unless $\nu\le_{\fp} \kappa$ and $(T^{\fp}(\kappa):\Delta^{\fp}(\kappa))=1$.
\item We have $(T^{\fp}(\kappa):\nabla^{\fp}(\nu)) =0$ unless $\nu\le_{\fp} \kappa$ and $(T^{\fp}(\kappa):\nabla^{\fp}(\kappa))=1$.
\item There exists a monomorphism $\Delta^{\fp}(\kappa)\hookrightarrow T^{\fp}(\kappa)$ with cokernel in $\cF(\Delta^{\fp})$.
\item There exists an epimorphism $ T^{\fp}(\kappa)\tto \nabla^{\fp}(\kappa)$ with kernel in $\cF(\nabla^{\fp})$.
\end{enumerate}
\end{enumerate}
\end{thm}
\begin{proof} 	We start with the embedding {$\Delta^{\fp_\oo}_\oo(\lambda)  \hookrightarrow T^{\fp_\oo}_{\oo}(\la)$ with cokernel admitting a $\Delta_{\oa}^{\fp_{\oa}}$-flag, where $T^{\fp_{\oa}}_\oo(\la)$ is the corresponding indecomposable tilting module in $\mc O(\mf g_\oo, \mf p_\oo)$}. This induces an embedding $\Ind \Delta_\oo^{\fp_\oo}(\lambda)  \hookrightarrow \Ind T^{\fp_\oo}_{\oo}(\la)$. We now observe that
\begin{align*}
&\Ind \Delta_\oo^{\fp_\oo}(\lambda) \cong \Ind_{\mf g_\oo}^{\mf g} \Ind_{\mf p_\oo}^{\mf g_\oo}L_{\mf l}(\lambda) \cong   \Ind_{\mf p}^{\mf g} \Ind_{\mf p_\oo}^{\mf p}L_{\mf l}(\lambda),
\end{align*} which means $\Ind \Delta_\oo^{\fp_\oo}(\lambda) $ has a $\Delta^{\mf p}$-flag starting at $\Delta^{\mf p}({\la+2\rho(\fu_{\ob})})$. Since $\Ind T^{\fp_\oo}_{\oo}(\la)$ has a flag of modules with subquotient of the form $\Ind  \Delta_\oo^{\fp_\oo}(\mu)$, starting at $\Ind  \Delta_\oo^{\fp_\oo}(\la)$, we may conclude that $\Ind T^{\fp_\oo}_{\oo}(\la)$ has a $\Delta^{\mf p}$-flag starting at $\Delta^{\mf p}({\la+2\rho(\fu_{\ob})})$.

 Similarly, using \eqref{eqBF}, we find that $\Ind T^{\fp_\oo}_{\oo}(\la)$ has a $\nabla^{\mf p}$-flag and hence is a tilting module.   Therefore there exists an indecomposable summand  $N$ of $\Ind T^{\fp_\oo}_{\oo}(\la)$ such that $N \in \cF(\Delta^{\fp})\cap\cF(\nabla^{\fp})$ having a $\Delta^{\mf p}$-flag starting at $\Delta^{\mf p}({\la+2\rho(\fu_{\ob})})$.

The above argument shows existence of the tilting module along with properties (ii)(a)--(c). We have also established that this module is contained in $\cT(\fg,\fp_{\oa})$ and that $\cT(\fg,\fp_{\oa})$ is contained in the category of tilting modules. The full result then follows from \cite[Theorem 4.2 in Section 4.5]{BS18} (also see, \cite[Section 5]{So} and \cite[Proposition 6.9]{CLW15}).
\end{proof}

It follows by definition, but also from Theorem~\ref{ThmTilt}(i) and Proposition~\ref{PropDD}, that $\bD$ restricts to a contravariant equivalence $\cT(\fg,\fp_{\oa})\to\cT(\fg,\hat{\fp}_{\oa})$. More precisely, we have the following lemma.

\begin{lem} \label{Lem::dualtilting}
	For any $\la \in X_{\fp_{\oa}}$, we have
${\bf D} T^{\mf p}(\la)  = T^{\hat \fp}(-w_0\la)$.
\end{lem}
\begin{proof}
	The proof follows from \eqref{ThmTilt::mono} of  Theorem \ref {ThmTilt} and Proposition \ref{PropDD}.
\end{proof}

\subsection{Ringel duality}
For a simple reflection $s\in W$, we have the right exact twisting functor $T_s$ on $\cO(\fg,\fb_{\oa})$ as in \cite[\S 4.3]{CC}, see also \cite{Ar2, AS}. Since these functors satisfy the braid relations, see, e.g., \cite{KM, CM2}, we have the twisting functor $T_{w_0}$ defined via composition with respect to an arbitrary reduced expression for $w_0$. We consider the cohomology functor $\cL_{\ell(w_0^{\fp})}T_{w_0}$ and denote by $\bT$ its restriction to $\cO(\fg,\fp_{\oa})$. It follows from \cite[Theorem~8.1]{CM3} that $\bT$ then factors as a (right exact) functor
$$\bT:\cO(\fg,\fp_{\oa})\to\cO(\fg,\hat\fp_{\oa}).$$
By \cite[equation~(5.1)]{CM2}, we have
\begin{equation}\label{TRes}
\bT\circ\Ind\cong\Ind\circ\bT\qquad\mbox{and}\qquad \Res\circ\bT\cong\bT\circ\Res,
\end{equation}
where we use the same notation for the functor $\bT$ as defined above for $\fg_{\oa}$.

\begin{thm} \label{Thm::2ndmain}
\begin{enumerate}[(i)]
\item The functor $\bT$ restricts to an equivalence $\cP(\fg,\fp_{\oa})\stackrel{\sim}{\to}\cT(\fg,\hat\fp_{\oa})$. For each $\fp\in\Par(\fg,\fp_{\oa})$,
\begin{align*}
\bD\bT P^{\fp}(\kappa)\;\cong\; T^{\fp}(-w_0^{\fp}\kappa+2\rho(\fu^-)),\quad\mbox{for all $\kappa\in X_{\fp_{\oa}}$.}
\end{align*}
\item For each $\fp\in\Par(\fg,\fp_{\oa})$, the functor $\bT$ restricts to an exact equivalence  $\cF(\Delta^{\fp})\stackrel{\sim}{\to}\cF(\nabla^{\hat\fp})$ with
$$\bT(\Delta^{\fp}(\kappa))\cong \nabla^{\hat\fp}(w_0w_0^{\fp}\kappa-2\rho(\hat{\fu}^+)),\quad\mbox{for all $\kappa\in X_{\fp_{\oa}}$.}$$
\item The highest weight categories $\cO(\fg,\fp)$ and $\cO(\fg,\hat{\fp})$ are Ringel dual in the sense of \cite[\S 4.5]{BS18}.
\item The functor $\bD\circ\bT$ restricts to a contravariant autoequivalence of $\cF(\Delta^{\fp})$ with
\begin{align*}
\bD\bT \Delta^{\fp}(\kappa)\;\cong\;\Delta^{\fp}(-w_0^{\fp}\kappa+2\rho(\fu^-)), \qquad \kappa\in X_{\fp_{\oa}}.
\end{align*}
\end{enumerate}
\end{thm}
\begin{proof}
By \cite[Section~8.1]{CM3}, the functor $\bT:\cO(\fg_{\oa},\fp_{\oa})\to\cO(\fg_{\oa},\hat{\fp}_{\oa})$, restricts to an equivalence $\cP(\fg_{\oa},\fp_{\oa})\stackrel{\sim}{\to}\cT(\fg_{\oa},\hat{\fp}_{\oa})$. Moreover, its inverse (as described in the proof of \cite[Theorem~8.1]{CM3}) is by \cite[Theorem~4.1]{AS} a cohomology functor of a completion functor.  In \cite[Section 4.2]{CC}, it is shown that completion functors can be defined on $\cO(\fg,\fp_{\oa})$ as well and satisfy the analogue of \eqref{TRes}. We denote the corresponding cohomology functor (both for $\fg$ and $\fg_{\oa}$) by $\mathbf{G}$.

By Remark~\ref{RemPI}, the functors $\bT$ and $\mathbf{G}$ restrict to functors between $\cP(\fg,\fp_{\oa})$ and $\cT(\fg,\fp_{\oa})$. By all the above, they satisfy
$$\bG\circ\bT\circ\Ind\cong\Ind \;\mbox{ on $\cP(\fg_{\oa},\fp_{\oa})$}\qquad\mbox{and}\qquad\bT\circ\bG\circ\Ind\cong\Ind \;\mbox{ on $\cT(\fg_{\oa},\fp_{\oa})$}.$$
It thus follows by definition of $\cT$ and Remark~\ref{RemPI} that $\bT$ is fully faithful on $\cP(\fg,\fp_{\oa})$ and each object in $\cT(\fg,\fp_{\oa})$ is a direct summand of an object in the image of $\bT$. It follows that $\bT$ yields an equivalence, which concludes the proof of part (i).

Now we compute the character of $\bD\bT(\Delta^{\fp}(\kappa))$, which is the character of $\Res \bD\bT(\Delta^{\fp}(\kappa))$. By \eqref{TRes} the functors $\bT$ for $\fg$ and $\fg_{\oa}$ intertwine $\Res$, and by construction the same is true for $\bD$.
Lemma~\ref{LemResDelta} implies that $\Res\Delta^{\fp}(\kappa)$ is in $\cF(\Delta^{\fp_{\oa}}_{\oa})$. By \cite[Section 8]{CM3} the character $\ch\bD\bT N$ with $N\in\cF(\Delta^{\fp_{\oa}}_{\oa})$ depends only on $\ch N$. We therefore observe that
$$\ch \bD\bT(\Delta^{\fp}(\kappa))\;=\; \ch \bD\bT\Ind^{\fg_{\oa}}_{{\fp}_{\oa}} ( L_{\fl}(\kappa)\otimes S(\fu^-_{\ob})),$$
where $S(\fu^-_{\ob})$ is to be interpreted as the $\fp_{\oa}$-module $S(\fg_{\ob}/\fp_{\ob})$.

Set $\rho_{\oa}=\rho(\fb_{\oa})$. By combining Proposition~\ref{PropDD} (applied to $\fg_{\oa}$) with $\bT\Delta^{\fp_{\oa}}(\nu)\cong \nabla^{\hat{\fp}_{\oa}}(w_0w_0^{\fp}(\nu+\rho_\oo)-\rho_\oo)$ \cite[Theorem~8.1]{CM3} we find that, with
\begin{align} \label{Eq::RootEq1}
&w_0\rho_\oo-w_0^{\fp}\rho_\oo=2\rho(\fu^-_{\oa}),
\end{align}
$$\bD\bT \left( \Ind^{\fg_{\oa}}_{{\fp}_{\oa}}L_{\fl}(\nu)\right)\cong\bD\bT \Delta^{\fp_{\oa}}(\nu)\cong \Delta^{\fp_{\oa}}(-w_0^{\fp}\nu+2\rho(\fu^-_{\oa}))\cong  \Ind^{\fg_{\oa}}_{{\fp}_{\oa}}(L_\fl(\nu)^\ast\otimes\mC_{2\rho(\fu^-_{\oa})}).$$
 By exactness of $\bD\bT$ we thus get in particular that
 $$\ch\bD\bT( \Ind^{\fg_{\oa}}_{{\fp}_{\oa}}M)\;=\; \ch\bD\bT( \Ind^{\fg_{\oa}}_{{\fp}_{\oa}}M^\ast\otimes\mC_{2\rho(\fu^-_{\oa})}),$$
 for any finite-dimensional $\fp_{\oa}$-module $M$.

 Combining the two above paragraphs thus finally shows
$$\ch  \bD\bT(\Delta^{\fp}(\kappa))\;=\;\ch \Ind^{\fg_{\oa}}_{\fp_{\oa}}(L_{\fl}(-w_0^{\fp}\kappa)\otimes S(\fu^-_{\ob})^\ast \otimes\mC_{2\rho(\fu^-_{\oa})} ).$$
Since
$$
\ch S(\fu^-_{\ob})^\ast = \ch( S(\fu^-_{\ob})\otimes \Pi^{\dim \fu^-_{\ob} }\mC_{2\rho(\fu_{\ob}^-)}),
$$
we thus conclude
\begin{equation}\label{eqcha}\ch  \bD\bT(\Delta^{\fp}(\kappa))\;=\;\ch\Delta^{\fp}(-w_0^{\fp}\kappa+2\rho(\fu^-)).\end{equation}

This now allows us to calculate the character of $\bD\bT P^{\fp}(\lambda)$ in terms of the multiplicities $(P^{\fp}(\lambda):\Delta^{\fp}(\nu))$, which vanish unless $\lambda\le_{\fp}\nu$. By part (i), $\bD\bT P^{\fp}(\lambda)$ must be an indecomposable tilting module. Theorem~\ref{ThmTilt}(a) and equation \eqref{eqcha} then allow to conclude
$$\bD\bT P^{\fp}(\lambda)\;\cong\; T^{\fp}(-w_0^{\fp}\lambda+2\rho(\fu^-)).$$

Furthermore, by the exactness of $\bD\bT$ on $\cF(\Delta^{\fp})$, we have an inclusion
$$\bD\bT \Delta^{\fp}(\lambda)\hookrightarrow T^{\fp}(-w_0^{\fp}\lambda+2\rho(\fu^-)).$$
Since we already understand the character of the left-hand module by \eqref{eqcha}, we know the module contains a vector of highest weight in the right-hand module. This means, by Theorem~\ref{ThmTilt}(ii)(c), that $\bD\bT \Delta^{\fp}(\lambda)$ contains $\Delta^{\fp}(-w_0^{\fp}\lambda+2\rho(\fu^-))$ as a submodule. However, since its character is equal to that of the parabolic Verma module, the modules are equal. Hence we have
$$\bD\bT \Delta^{\fp}(\lambda)\;\cong\;\Delta^{\fp}(-w_0^{\fp}\lambda+2\rho(\fu^-)).$$
By Lemma~\ref{LemResDelta}(iii), it thus follows that $\bD\bT$ restricts to an exact functor $\cF(\Delta^{\fp})\to \cF(\Delta^{\fp})$. That this is an equivalence follows as in the first paragraph. This proves part (iv), and part (ii) now follows using Proposition~\ref{PropDD}.
\end{proof}

\begin{cor} \label{Cor::ASdual} We have	
\begin{eqnarray*}
&&(T^\fp(-w_0^{\fp}\lambda+2\rho(\fu^-)): \Delta^\fp(-w_0^{\fp}\mu+2\rho(\fu^-)))=[\nabla^\fp(\mu):L^\fp(\la)], \text{or equivalently,} \label{Eq1::Cor::ASdual}  \\
&&(T^\fp(\lambda): \Delta^\fp(\mu))=[\nabla^\fp(-w_0^\fp\mu+2\rho(\fu^-)):L^\fp(-w_0^\fp\la+2\rho(\fu^-))].\label{Eq2::Cor::ASdual}
\end{eqnarray*}

\end{cor}
\begin{proof}
	This is an immediate consequence of Lemma~\ref{LemResDelta}(ii) and Theorem~\ref{Thm::2ndmain}(i) and (iv).
\end{proof}

\begin{rem} Set $\rho^-=\rho(\fb^-)$. If follows easily that we can rewrite Corollary~\ref{Cor::ASdual} as
 $$(T^\fp(\la+\rho^-): \Delta^\fp(\mu+\rho^-))=[\nabla^\fp(-w_0^\fp\mu+\rho^-) :L^\fp(-w_0^\fp\la+\rho^-)].$$
\end{rem}

\begin{rem}
In \cite{Br04, So} Ringel duality and its application to the characters of tilting modules is studied in extensive generality. Concretely, they consider $\mathbb Z$-graded Lie (super)algebras, not necessarily finite-dimensional, which are generated by their subspaces in degree $-1, 0,1$. The parabolic subalgebra is then spanned by the non-negatively graded subspaces. However, the theory in {\it loc.~cit.} does not apply to $\mathfrak{pe}(n)$, since the $\mZ$-gradings on $\mathfrak{pe}(n)$ corresponding to its parabolic subalgebras do not satisfy the above condition. The formulas in Corollary \ref{Cor::ASdual} allow us to extend the results of \cite{Br04, So} to a generality which includes periplectic Lie superalgebras.
\end{rem}


\section{Projective-injective modules}\label{SecPI}

\subsection{Parabolic category $\cO$}
Fix a classical Lie superalgebra $\fg$.
Set $\eta:=2\rho(\fg)=2\rho(\fg_{\ob})=(\eta',i) \in \mf h^\ast\times\mZ_2$, so that $S^{top}\fg_{\ob}\cong \Pi^i\C_{\eta'}$ as  $\mf g_\oo$-modules.

 Now we fix a reduced parabolic subalgebra $\fp$ of $\fg$.

\begin{lem} \label{Cor::onedimequiv}
We have an auto-equivalence $-\otimes L(\eta)$ on   $\mc O(\mf g, \fp_{\oa})$, which satisfies $L(\mu)\mapsto L(\mu+\eta)$ for all $\mu\in X_{\fp_{\oa}}$, and has inverse $-\otimes L(-\eta)$. Furthermore, $\Ind\cong (L(\eta)\otimes -)\circ\Coind$.
\end{lem}
\begin{proof}
This follows immediately from Corollary~\ref{CorEta} and equation~\eqref{eqBF}.
\end{proof}

\begin{lem}\label{lem::Mmunu}  Let $\nu \in X_{\fp_{\oa}}$. Then there exist $\{m_{\mu,\nu}\in\mN\,|\,\mu \in X_{\fp_{\oa}}\}$ with $m_{\mu, \nu} =0$ unless $\mu \geq_{\fb} \nu$ and $m_{\nu,\nu}=1$ such that
\begin{align*}
&\Ind P_\oa^{\fp_{\oa}}(\nu) = \bigoplus_\mu P^\fp(\mu)^{\oplus m_{\mu,\nu}}, ~\Coind I_\oa^{\fp_{\oa}}(\nu) = \bigoplus_\mu I^\fp(\mu)^{\oplus m_{\mu,\nu}}.
\end{align*}
\end{lem}
\begin{proof}
	The observation $$\dim\text{Hom}_\g (\Ind P_\oa^{\fp_{\oa}}(\nu), L^\fp(\mu))  = [\Res L^\fp(\mu): L^{\fp_{\oa}}_\oa(\nu)] =\dim \text{Hom}_\g (L^\fp(\mu), \Coind I_\oa^{\fp_{\oa}}(\nu))$$
	implies the lemma.
\end{proof}

\begin{lem}\label{lem::chPI}
	For any $\mu \in X_{\fp_{\oa}}$ we have
	$\ch P^\fp(\mu) = \ch I^\fp(\mu+\eta). $\end{lem}
\begin{proof}
	By Lemma \ref{lem::Mmunu} we have the following character formulas
	\begin{align*}
	&\ch \Ind P_\oa^{\fp_{\oa}}(\nu) = \sum_{\mu \geq_{\fb} \nu} m_{\mu,\nu}\ch P^\fp(\mu), ~\ch  \Coind I_\oa^{\fp_{\oa}}(\nu) =\sum_{\mu \geq_{\fb} \nu} m_{\mu,\nu}\ch I^\fp(\mu).
	\end{align*}
	
	Let $(n_{\mu,\nu})_{\mu,\nu}$ be the inverse matrix of $(m_{\mu,\nu})_{\mu,\nu}$. Then we have the following expression of characters
	\begin{align*}
	&\ch P^\fp(\mu) = \sum_{\mu\geq_{\fb} \nu}n_{\mu,\nu}\ch\Ind P_\oo^{\fp_{\oa}}(\nu),~\ch I^\fp(\mu) = \sum_{\mu\geq_{\fb} \nu}n_{\mu,\nu}\ch\Coind I_\oo^{\fp_{\oa}}(\nu).
	\end{align*}
 Now by Lemma \ref{Cor::onedimequiv} we have $n_{\mu,\nu} =  n_{\mu+\eta,\nu+\eta}$, for every $\mu,\nu\in X_{\fp_\oo}$, and
		\begin{align*}
		&\text{ch}P^{\fp}(\mu) = \sum_{\mu\geq_{\fb} \nu} n_{\mu + \eta,\nu +\eta}\text{ch} \Ind P_\oo^{\fp_{\oa}}(\nu) = \sum_{\mu\geq_{\fb} \nu} n_{\mu+\eta,\nu+\eta}\text{ch} \Coind I_0^{\fp_{\oa}} (\nu+\eta) = \text{ch} I^{\fp} (\mu +\eta).
		\end{align*}
		This completes the proof.
\end{proof}

The classification of projective-injective modules in $\mc O(\mf g_\oo, \fp_{\oa})$ is studied in  \cite{Ir} and \cite[Section 5]{MS}. {For the  category $\mc O$ of classical Lie superalgebras with simple-preserving duality, the classification is given in \cite{Ma}. We are now in a position to prove the following generalization.}

\begin{thm} \label{prop::paraIrving} For any $\la \in X_{\fp_{\oa}}$ the following assertions are equivalent:
	\begin{enumerate}
		\item $P^\fp(\la)$ is  injective in $\mc O(\mf g, \fp_{\oa})$.\label{prop::paraIrving::Eq2}
		\item $P^\fp(\la) =I^\fp(\la+\eta)$. \label{prop::paraIrving::Eq4}
		\item $I^\fp(\lambda)$ is projective in $\mc O(\mf g, \fp_{\oa})$. \label{prop::paraIrving::Eq9}
		\item $P^\fp(\la)\in \cT(\fg,\fp_{\oa})$.   \label{prop::paraIrving::Eq3} 
		\item $L^\fp(\la)$ occurs in the socle of a parabolic Verma module. \label{prop::paraIrving::Eq5}
		\item $[\Res L^\fp(\la):L^{\fp_{\oa}}_\oo(\mu)]\neq 0$ for some $\mu$ such that $P_\oo^{\fp_\oa}(\mu)$ is injective in $\mc O(\mf g_\oo, \fp_{\oa})$.\label{prop::paraIrving::Eq8}
	\end{enumerate}
\end{thm}

\begin{proof}
That \eqref{prop::paraIrving::Eq2} implies \eqref{prop::paraIrving::Eq4} follows from the observation that the characters of indecomposable injective modules are linearly independent (which is a direct consequence of Lemma~\ref{LemResDelta}(ii)) and Lemma~\ref{lem::chPI}.

That \eqref{prop::paraIrving::Eq4}  implies \eqref{prop::paraIrving::Eq9} follows from Lemma~\ref{Cor::onedimequiv}.

That  \eqref{prop::paraIrving::Eq9} implies  \eqref{prop::paraIrving::Eq5} follows from the fact that projective modules are in $\cF(\Delta^{\fp})$.

Now we show that \eqref{prop::paraIrving::Eq5} implies \eqref{prop::paraIrving::Eq8}. It follows from Lemma~\ref{LemResDelta}(iii) that \eqref{prop::paraIrving::Eq5} implies that the simple modules in the socle of $\Res L^{\fp}(\la)$ are in the socle of parabolic Verma modules. By the equivalence of  \eqref{prop::paraIrving::Eq5} and \eqref{prop::paraIrving::Eq2} for the reductive Lie algebra $\fg_{\oa}$, property \eqref{prop::paraIrving::Eq8} follows.

Now we show that \eqref{prop::paraIrving::Eq8} implies \eqref{prop::paraIrving::Eq2}. We have a non-zero morphism $\Ind P_{\oa}^{\fp_{\oa}}(\mu)\to L^{\fp}(\la)$, so $P^\fp(\la)$ is a direct summand of  $\Ind P_{\oa}(\mu)$. By \eqref{eqBF}, $\Ind P_{\oa}^{\fp_{\oa}}(\mu)$ is injective, which implies \eqref{prop::paraIrving::Eq2}.

It now only remains to show that \eqref{prop::paraIrving::Eq3} is equivalent to the other conditions. That \eqref{prop::paraIrving::Eq2} implies \eqref{prop::paraIrving::Eq3} is clear, since projective modules are in $\cF(\Delta^{\fp})$ and injective modules in $\cF(\nabla^{\fp})$. 

Finally, we show that \eqref{prop::paraIrving::Eq3} implies \eqref{prop::paraIrving::Eq8}. If \eqref{prop::paraIrving::Eq3} is satisfied, then $L^{\fp}(\lambda)$ appears in the top of a module $\nabla^{\fp}(\mu)$. We can now proceed as in our proof of the fact that \eqref{prop::paraIrving::Eq5} implies \eqref{prop::paraIrving::Eq8}.
\end{proof}

\begin{rem}\label{RemPI}
The equivalence of \eqref{prop::paraIrving::Eq2} and \eqref{prop::paraIrving::Eq8} in Theorem~\ref{prop::paraIrving} shows that every projective-injective module in $\cO(\fg,\fp_{\oa})$ is a direct summand of a module $\Ind M$ with $M$ projective-injective in $\cO(\fg_{\oa},\fp_{\oa})$. On the other hand, by \eqref{eqBF} it is clear that for $M$ projective-injective in $\cO(\fg_{\oa},\fp_{\oa})$, the module $\Ind M$ is again projective and injective. Hence the subcategory of projective-injective modules in $\cO(\fg,\fp_{\oa})$ is the same as the subcategory of direct summands of modules $\Ind M$ for $M$ projective-injective in $\cO(\fg_{\oa},\fp_{\oa})$.
\end{rem}

By a simple-preserving duality on a category we mean a contravariant equivalence which is isomorphic to the identity functor on the full subcategory of simple objects.
\begin{cor}
If $\eta\neq 0$ (i.e. $S^{top}\fg_{\ob}\not\cong\mC$ as a $\mZ_2$-graded $\fg_{\oa}$-module) the category $\mc O(\fg,\fp_{\oa})$ does not admit a simple-preserving duality.	
\end{cor}
\begin{proof} It follows from Remark~\ref{RemPI} and the results for Lie algebras in \cite{Ir} that $\cO$ contains projective-injective modules. Let $P^\fp(\la)$ be injective.
	Assume that $\mc O$ admits a simple-preserving duality $\tau$ then $\ch \tau P^\fp(\la) =\ch P^\fp(\la)$ implies that
	 $ \tau P^\fp(\la) = P^\fp(\la)$. By Theorem \ref{prop::paraIrving} we have $$L^\fp(\la+\eta) = \text{soc}I^\fp(\la+\eta) = \text{soc} P^\fp(\la) = L^\fp(\la),$$
	 and therefore $\eta =0$.
\end{proof}

\begin{cor} \label{Cor::typeIPI} Let $\la \in X_{\fp_\oo}$.  If $P^{\fp_\oo}_\oo(\la)$ is injective in $\mc O(\mf g_\oo, \fp_\oo)$ then $P^{\fp}(\la)$ is injective in  $\mc O(\mf g, \fp_\oo)$.
	
	Furthermore, assume that $\mf g$ admits a $\mathbb Z$-grading $\mf g= \mf g_{-1}\oplus \mf g_0 \oplus \mf g_1$, compatible with its $\mZ_2$-grading.
	If $\fp=\fp_{\oa}\oplus\fg_1$, then $P^{\mf p}(\la)$ is injective if and only if $P_\oo^{\mf p_\oo}(\la)$ is injective.
\end{cor}
\begin{proof} We first note that $[\Res L^{\fp}(\la): L_\oo^{\fp_{\oa}}(\la)]\neq 0$. Therefore the first claim follows from the equivalence of \eqref{prop::paraIrving::Eq2} and \eqref{prop::paraIrving::Eq8} in Theorem \ref{prop::paraIrving}.
	
We now proceed with the proof of the second claim.

Note that a highest weight vector in $L^{\fp}(\la)$ generates a $\fg_{\oa}$-submodule isomorphic to $L^{\fp_{\oa}}(\la)$ and the $\fg_{\oa}$-submodule $M=U(\fg_{-1})\fg_{-1}L^{\fp_{\oa}}(\la)$ does not contain any highest weight vector. Since we have a surjection $\wedge(\fg_{-1})\otimes L^{\fp_{\oa}}(\la)\twoheadrightarrow \Res L^{\fp}(\la)$, it follows that $\Res L^{\fp}(\la)/M\cong L^{\fp_{\oa}}(\la)$.
{}
In particular, we have a surjection
$$\Res L^{\fp}(\lambda)\tto L^{\fp_{\oa}}_{\oa}(\la).$$
Since $\Res$ is exact it follows that the projective module $\Res P^{\fp}(\la)$ contains $P^{\fp_{\oa}}_{\oa}(\la)$ as a direct summand. Since the left-adjoint functor to $\Res$ is exact, it follows that
$\Res P^{\fp}(\la)$ is injective when $P^{\fp}(\la)$ is. This concludes the proof.
\end{proof}

\begin{rem}
The natural bijection, given in Corollary \ref{Cor::typeIPI}, of projective-injective modules between $\mc O(\mf g, \fp_{\oa})$ and $\mc O(\mf g_\oo, \fp_{\oa})$ of $\mf g$ outside type I fails in general, see Section \ref{Sect::FullO}.
\end{rem}


Let $\mf b$ be a Borel subalgebra.  The following proposition determines the highest weights of injective projective covers.
\begin{cor}\label{Cor::ProTilt} Assume that $\la\in\fh^\ast$ with $(\lambda, \alpha^\vee )\in\mZ$ for every $\alpha\in\Phi^+_{\bar 0}$.
	If the projective cover $P^{\mf b}(\lambda)$ is injective, then  we have $$P^{\mf b}(\la) = T^{\hat{\mf b}}(w_0\la -2\rho(\hat{\fb})).$$
\end{cor}
\begin{proof} 
	By \cite[Lemma 5.16]{CM2}, the functor $\bT$ sends every integral projective-injective module to itself. By Theorem \ref{Thm::2ndmain}(i) and Lemma \ref{Lem::dualtilting} we have
	\begin{align*}
	&P^{\mf b}(\la) = {\bf T} P^{\mf b}(\la) ={ \bf D}T^{\mf b}(-\la +2\rho(\fb^-)) = T^{\hat{\mf b}}(w_0\la -w_02\rho(\fb^-)) = T^{\hat{\mf b}}(w_0\la -2\rho(\hat{\fb})),
	\end{align*} as desired. 
\end{proof}

\subsection{Full category $\cO$: general statement} \label{Sect::FullO}

Let $\fg$ be a finite-dimensional classical Lie superalgebra with triangular decomposition induced by a regular element $H\in\fh_{\oa}$ as in Section \ref{DefBPS}:
$$\g=\fn^+\oplus\h\oplus\fn^-,$$
with Borel subalgebra $\fb=\fh\oplus\fn^+$ of $\fg$.
Let $\Phi^+=\Phi(\fb)$ be the set of roots in $\fn^+$, with simple system $\Pi$. Let $\Phi^+_{\bar 0}$ be the subset of positive even roots, which is a root system of a semisimple Lie algebra. Let $\Pi(\fn^+_{\bar 0})$ be the corresponding simple system for $\Phi^+_{\bar 0}$. Note that in general not every element of $\Pi(\n^+_{\bar 0})$ is an element in $\Pi$. For every $\alpha\in \Pi(\fn^+_{\bar 0})$, we choose non-zero vectors $f_\alpha\in \fg_{\oa}^{-\alpha}$ and $e_\alpha\in\fg_{\oa}^\alpha$.  Let $L(\la)=L^{\fb}(\la)$ be the irreducible highest weight module of highest weight $\la\in\fh_{\oa}^\ast$ with respect to this triangular decomposition.
We fix a $W$-invariant non-degenerate bilinear form $(\cdot,\cdot)$ on $\fh_{\oa}^\ast$, which we assume to be induced from an invariant non-degenerate bilinear form on $\fg$ if the latter exists. For $\beta\in \Phi^+_{\oa}$ we set $\beta^\vee=2\beta/(\beta,\beta)$.

For clarity of exposition we will restrict to integral weights, that is, weights $\la\in\fh^\ast$ with $(\lambda, \alpha^\vee )\in\mZ$ for every $\alpha\in\Phi^+_{\bar 0}$. 

\begin{thm}\label{thm:char:proinj}
For every integral $\la\in\fh^\ast_{\oa}$, the projective cover $P(\la)$ is injective if and only if $L(\la)$ is a free $U(f_\alpha)$-module for every $\alpha\in\Pi(\fn^+_{\bar 0})$.
\end{thm}

We need some preparatory results before proving the theorem.
The following lemma is well-known. We add a proof for the reader's convenience.
\begin{lem}
Let $\fk$ be a finite-dimensional Lie superalgebra and $X\in\fk_{\oa}$. For every simple $\fk$-module $L$, the element $X$ either acts freely on $L$ or else $X$ acts locally finitely on $L$.
\end{lem}
\begin{proof}
Suppose that $X$ does not act freely on $L$. Then there exists a nonzero vector $v\in L$ and a positive integer $m$ such that $X^m v=0$. Since $U(\fk)v=L$ and $X$ acts locally finitely on $\fk$, we see that $X$ acts locally finitely on $L$. \end{proof}

	\begin{lem}\label{lemanti}
Let $\g$ be a classical Lie superalgebra possessing a $\mathbb Z$-grading $\g=\bigoplus_{j\in \mathbb Z}\mf g_j$ such that $\g_{\bar 0}=\bigoplus_{j\in \mathbb Z}\g_{2j}$ and $\g_{\bar 1}=\bigoplus_{j\in \mathbb Z}\g_{2j+1}$. Let $\fb_0$ be a Borel subalgebra of the reductive Lie algebra $\fg_0$ so that $\fb=\fb_0\oplus\bigoplus_{j>0}\fg_j$ is a Borel subalgebra for $\fg$. If $\alpha$ is a simple root of $\fg_{0}$, then $L(\lambda)$ is $U(f_\alpha)$-free if and only if $(\lambda,\alpha^\vee)\not\in \mathbb N$.
\end{lem}
\begin{proof}

Let $v$ be a highest weight vector of $L(\la)$. By the ordinary $\mathfrak{sl}(2)$-relations, for $m\in\mN$, we have
$$e_\alpha^m f_\alpha^m v = 0 \quad\mbox{if and only if } m>(\lambda, \alpha^\vee) \in\mN.$$
Hence, if $( \la,\alpha^\vee)\not\in\mN$, the module $L(\la)$ is clearly $f_\alpha$-free.

Now assume that $n:=( \la,\alpha^\vee)\in\mN$. For any $j>0$, the $\fg_0$-module $\mf g_j$ is a direct sum of irreducible $\fg_0$-modules. Let $w_{{j1}},\ldots,w_{{jk_j}}$ be the set of all $\fb_0$-lowest weight vectors in $\fg_j$. By definition we have $[w_{{ji}},f_\alpha]=0$ for all $i$. Therefore we have $w_{{ji}} f^{n+1}_\alpha v=0$ for all $i,j$. It follows that $f_{\alpha}^{n+1}v$ is annihilated by all simple roots of $\fg_0$ and all vectors of $\fg_j$, for $j>0$. Since $L(\la)$ is simple, we have $f^{n+1}_\alpha v=0$ and so $L(\la)$ is $f_\alpha$-finite.
\end{proof}

\begin{proof}[Proof of Theorem~\ref{thm:char:proinj}]
Assume that $P(\la)$ is injective. By Theorem~\ref{prop::paraIrving} the simple module $L(\lambda)$ is in the socle of some $U(\fn^-)$-free module $\Delta(\mu)$. In particular $L(\lambda)$ is itself $f_\alpha$-free for every $\alpha\in\Phi_{\oa}^+$.

Now we prove the other direction of the claim for the case $\fg=\fg_{\oa}$, so in particular $\Pi=\Pi(\fn^+_{\oa})$. If $L(\la)$ is $f_\alpha$-free for all $\alpha\in\Pi$, then $\la$ is anti-dominant by Lemma \ref{lemanti}. Consequently $L(\la)=\Delta(\la)$ and the claim follows from Theorem~\ref{prop::paraIrving}.

Now we go back to the case where $\fg$ is a superalgebra. If $L(\la)$ is $f_\alpha$-free for all $\alpha\in\Pi(\fn^+_{\oa})$, then so is any simple module in the socle of $\Res L(\la)$. By the above observation for $\fg_{\oa}$ and equivalence of (1) and (6) in  Theorem~\ref{prop::paraIrving}, it follows that $P(\la)$ is injective.
\end{proof}

\begin{rem}
Observe that this type of characterisation is not possible in a general parabolic category $\cO$. Already for $\fg=\gl(4)$ with a parabolic subalgebra with Levi subalgebra $\gl(2)\oplus\gl(2)$, we cannot characterise simple modules with projective-injective cover based on freeness of the action of root vectors.
\end{rem}

\subsection{Full Category $\cO$: the simple classical Lie superalgebras}
\subsubsection{Type I} \label{subsubSect::typeI}
We say that a classical Lie superalgebra $\g$ is of type I, if it has  a $\mathbb Z$-gradation of the form $\g= \g_{-1}\oplus \mf g_0 \oplus \g_1$ with $\g_{\bar 0}=\g_0$, $\g_{\bar 1}=\mf g_{-1}\oplus \g_1$. Note that such a $\mathbb Z$-gradation together with a triangular decomposition of $\fg_{\bar 0}$ gives rise to a natural triangular decomposition $\fg=\fn^-+\fh+\fn^+$.   We note that here $\mf h+\mf n^+$ and $\mf h+ \mf n^-$ may not be Borel subalgebras defined as in Section \ref{DefBPS}, see, e.g., $\fg=sl(1|1)$.

Examples of  classical Lie superalgebras of type I  are
\begin{align}\label{Eq::typeIsuper}
&\mathfrak{gl}(m|n),\;\,\mathfrak{sl}(m|n),\;\, \mathfrak{psl}(n|n),\;\,\mathfrak{osp}(2|2n),\;\mf{pe}(n),\;\mf{spe}(n)=[\mf{pe}(n),\mf{pe}(n)],
\end{align}
see \cite{ChWa12} and \cite{Mu12} for more details. Other examples of type I classical Lie superalgebras are provided by Examples \ref{classical:semi1} and \ref{ExSDP}(2).


The classification of projective-injective modules in category $\cO$ for Lie superalgebras of type I is well-known, see e.g. Corollary~\ref{Cor::typeIPI}. For completeness we include it here again in a slightly different formulation.
An integral weight $\la\in \mf h_\oo^\ast$ is said to be anti-dominant if $(\la,\alpha^\vee)\not\in\mN$, for any even simple root $\alpha$. A direct application of Theorem \ref{thm:char:proinj} and Lemma~\ref{lemanti}
yields the following.

\begin{prop} \label{Thm::claPrjInjtypeI}
Let $\g$ be of type I and let $\g=\n^++\h+\n^-$ be the triangular decomposition above and let $P(\la)$ be the projective cover in $\cO(\fg,\fb)$ of the simple module $L^{\fb}(\la)$. Then $P(\la)$ is injective if and only if $\la$ is anti-dominant.
\end{prop}

	

\subsubsection{The case $\mathfrak{q}(n)$}Consider the standard triangular decomposition of $\g=\mathfrak{q}(n)=\n^++\h+\n^-$. Note that the Cartan subalgebra $\h=\h_{\bar 0}\oplus\h_{\bar 1}$ is not abelian. Let $H_{1},\ldots,H_{n}$ be the standard basis for $\h_{\bar 0}$ with dual basis $\{\epsilon_1,\ldots,\epsilon_n\}$ of $\fh_{\oa}^\ast$ (see \cite[Section 1.2.6]{ChWa12} for notation and precise definition).

We recall the following well-known fact.

\begin{prop}\label{prop::q2}
Let $\g=\mathfrak{q}(2)$. Set $\alpha=\ep_1-\ep_2$ and {$\la=m\epsilon_1+n\epsilon_2\in\h^*_{\bar 0}$.}
\begin{itemize}
\item[(i)] If {$m\not=0$, then $L(m\epsilon_1+m\epsilon_2)$ is $U(f_\alpha)$-free.}
\item[(ii)] The root vector $f_\alpha$ acts trivially on $L(0)$.
\item[(iii)] If {$m-n\in \mZ_{>0}$, then  $L(m\epsilon_1+n\epsilon_2)$ is $U(f_\alpha)$-finite.}
\item[(iii)] If {$m-n\notin\mZ_{>0}$, then  $L(m\epsilon_1+n\epsilon_2)$ is $U(f_\alpha)$-free.}
\end{itemize}
\end{prop}

\begin{proof}
Let $v$ be a highest weight vector in $L(m\epsilon_1+m\epsilon_2)$. When $m\not=0$, it can be shown that $U(f_\alpha)v$ is a free $U(f_\alpha)$-submodule (see, e.g., \cite[Lemma 2.17]{ChWa12}), so $L(m\epsilon_1+m\epsilon_2)$ is free over $U(f_\alpha)$.  Also, it is well-known that if $m-n\in \mZ_{>0}$ then $L(m\epsilon_1+n\epsilon_2)$ is finite dimensional, see, e.g., \cite[Theorem 2.18]{ChWa12}. This completes the proof of $\text{(i), (ii)}$ and $\text{(iii)}$.

Finally, if $m-n\notin \N\backslash \{0\}$ then $U(f_\alpha)v$ is a free $U(f_\alpha)$-submodule by the standard $\mathfrak{sl}(2)$-theory. This completes the proof.
\end{proof}

\begin{thm} Let $\g=\mathfrak{q}(n)$ with standard triangular decomposition $\g=\n^++\h+\n^-$. Let $\la=\sum_{i=1}^n\la_i\ep_i\in\h^*_{\bar 0}$ be integral, then
$P(\lambda)$ is injective if and only if $\lambda_1\le\lambda_2\le\ldots\le\lambda_n$ and in addition if $\la_i=\la_{i+1}$, then $\la_i\not=0$.
\end{thm}

\begin{proof}
By Proposition \ref{prop::q2}, we see that $f_\alpha$ acts freely on any nonzero highest weight vector of $L(\la)$, for all $\alpha\in\Pi(\fn^+_{\oa})$, if and only if $\la$ is of the form as in the theorem. The conclusion follows from Theorem \ref{thm:char:proinj}.
\end{proof}

\subsubsection{The case $\mathfrak{spo}(2n|2m)$} \label{subsubsect::spo2n2m}
Since the notion of odd reflection for a contragredient Lie superalgebra plays a crucial role below, we shall briefly recall it for the convenience of the reader in the next paragraph.

Given a simple system $\Pi$ with positive roots $\Phi^+$ of a contragredient Lie superalgebra and
a simple isotropic odd root $\gamma\in\Pi$ one can construct a new simple system $r_\gamma\Pi=\Pi'$ with $-\gamma\in \Pi'$ and positive roots $\Phi^+\setminus\{\gamma\}\cup\{-\gamma\}$. If $\fb$ is the Borel subalgebra associated with $\Phi^+$ and $\fb'$ is the one associated with the new system, and $L^\fb(\la)$ is the irreducible module of $\fb$-highest weight $\la$, then the $\fb'$-highest weight of $L^\fb(\la)$ is as follows: It equals $\la$, if $(\la,\gamma)=0$, and it equals $\la-\gamma$, if $(\la,\gamma)\not=0$, see \cite[Lemma 1]{PS89}. Below we shall freely use this fact, referring the reader to op.~cit. or \cite[Sections 1.4--1.5]{ChWa12} for more details.

Let $\g=\mathfrak{spo}(2n|2m)$ and consider the triangular decomposition determined by the following simple system:
\begin{align}\label{Pi::ospeven}
\Pi=\{\delta_1-\delta_2,\ldots,\delta_{n-1}-\delta_n,\delta_n-\epsilon_1,\epsilon_1-\epsilon_2,\ldots,\epsilon_{m-1}-\epsilon_{m}, \epsilon_{m-1}+\epsilon_{m}\}.
\end{align}
Denote by $\fb$ the Borel subalgebra corresponding to the simple system \eqref{Pi::ospeven}. Let $\fb=\fh+\fn^+$ and $\rho=\rho(\fn^+)$.
Also, recall that the bilinear form is given by $(\delta_i|\delta_j)=-(\epsilon_i|\epsilon_j)=\delta_{ij}$ and $(\delta_i|\epsilon_j)=0$, see, e.g., \cite[Section 1.2.2]{ChWa12} for notation and further details.

\begin{thm}\label{thm:injpro:d}
Let $\lambda=\sum_{i=1}^n\lambda_i\delta_i+\sum_{j=1}^m\mu_j\epsilon_j\in\h^*$ be integral. Then the projective cover $P(\la)$ is injective if and only if $\lambda_1<\lambda_2<\ldots<\lambda_n<m$ and $\mu_1<\mu_2<\ldots<\mu_{m-1}<-|\mu_m|$ and, in case $\la_n=m-1$, then in addition  $\mu_m\not=0$.
\end{thm}

\begin{proof}
Note that
\begin{align*}
\Pi(\fn_{\oa}^+)=\{\delta_1-\delta_2,\ldots,\delta_{n-1}-\delta_n,2\delta_n,\ep_1-\ep_2,\ldots,\ep_{m-1}-\ep_{m},\ep_{m-1}+\ep_m\}.
\end{align*}
By Theorem~\ref{thm:char:proinj} and Lemma~\ref{lemanti}, $P(\la)$ is injective if and only if $\la$ satisfies
\begin{align}\label{cond:basic:osp}
\la_1<\la_2<\cdots<\la_n,\quad \mu_1<\mu_2<\cdots<\mu_{m-1}<-|\mu_m|.
\end{align}
and $L^\fb(\la)$ is $f_{2\delta_n}$-free.  To complete the proof we need to show that the additional condition in the statement ($\lambda_n<m-1$ or $\lambda_n=m-1$ with $\mu_m\not=0$) is necessary and sufficient for $L^\fb(\la)$ to be $f_{2\delta_n}$-free.

Consider the sequence of odd isotropic roots
\begin{align*}
\delta_n-\ep_1,\delta_n-\ep_2,\ldots,\delta_n-\ep_m.
\end{align*}
Let us denote the highest weight of $L^{\fb}(\la)$ with respect to $r_{\delta_n-\ep_j}\cdots r_{\delta_n-\ep_1}\Pi$ by $\la^{[j]}$, for $j=1,\ldots,m$, and set  $\la^{[0]}=\la$. Applying the sequence of odd reflections $r_{\delta_n-\ep_m}\cdots  r_{\delta_n-\ep_1}$ to $\Pi$, we see that the simple system $r_{\delta_n-\ep_m}\cdots r_{\delta_n-\ep_1}\Pi$ contains the even root $2\delta_n$ as a simple root. Note also that $(2\delta_n)^\vee=\delta_n$. In conclusion, by  Lemma~\ref{lemanti}, $L^\fb(\la)$ is $f_{2\delta_n}$-free if and only if $\la^{[m]}_n=(\la^{[m]},\delta_n)<0$.

Now we go through the different possibilities for $\lambda$ satisfying \eqref{cond:basic:osp}.

First assume that $\lambda_n\ge m$. By construction we have $\lambda^{[m]}_n\ge \lambda_n-m$. So in this case $\lambda^{[m]}_n\ge 0$ and $L^\fb(\la)$ is not $f_{2\delta_n}$-free.

Next assume that $\lambda_n=m-1$ and $\mu_m=0$. Since by construction $\lambda^{[m-1]}_n\ge \lambda_n-m+1=0$, the property $\mu_m=0$ implies that also $\lambda_n^{[m]}\ge 0$, so $L^\fb(\la)$ is not $f_{2\delta_n}$-free.

Next assume that $\lambda_n=m-1$ and $\mu_m\not=0$. In this case $\mu_{m-i}<-i$ for $i>0$ (by \eqref{cond:basic:osp}) and it follows easily that $\lambda^{[m-1]}_n=0$ and $\lambda^{[m]}_n=-1$, so $L^\fb(\la)$ is $f_{2\delta_n}$-free.

Finally assume that $\lambda_n<m-1$. Since $\mu_{m-i}\le-i$ for $i>0$, we find $\lambda_n^{[m-1]}=\lambda_n-m+1<0$, so certainly $\lambda_n^{[m]}<0$.
\end{proof}




In terms of $\rho$-shifted weights we can formulate Theorem \ref{thm:injpro:d} as follows:
\begin{cor}\label{cor:aux1}
The projective cover $P(\la)$ is injective if and only if $(\la+\rho,\alpha^\vee)\le 0$ for all $\alpha\in\Pi(\fn_{\oa}^+)$ and, in case $(\la+\rho,\delta_n)=0$, then in addition $(\la+\rho,\ep_m)=0$.
\end{cor}

\begin{rem}
The additional condition in Corollary \ref{cor:aux1} when $(\la+\rho,\delta_n)=0$ in the case $n=1$ can be shown to be equivalent to $\la$ being typical.
\end{rem}

\subsubsection{The case of $\mathfrak{spo}(2n|2m+1)$}

Let $\g=\mathfrak{spo}(2n|2m+1)$ and consider the the triangular decomposition determined by the following simple system:
\begin{align*} \{\delta_1-\delta_2,\ldots,\delta_{n-1}-\delta_n,\delta_n-\epsilon_1,\epsilon_1-\epsilon_2,\ldots,\ep_{m-1}-\ep_m,\epsilon_{m}\}.
\end{align*}
Denote by $\fb$ the Borel subalgebra corresponding to the simple system above.

\begin{thm}\label{thm:injpro:b}
Let $\lambda=\sum_{i=1}^n\lambda_i\delta_i+\sum_{j=1}^m\mu_j\epsilon_j\in\h^*$. Then the projective cover $P(\la)$ is injective if and only if $\lambda_1<\lambda_2<\ldots<\lambda_n<m$ and $\mu_1<\mu_2<\ldots<\mu_m<0$.
\end{thm}

\begin{proof}[Sketch of a proof]
The proof of the theorem above goes along the line as the proof of Theorem \ref{thm:injpro:d} and we shall only give the main ingredients below. The omitted computations here are simpler than those for $\frak{spo}(2m|2n)$ given in the proof of Theorem \ref{thm:injpro:d}.

Here we also use the sequence of odd reflections corresponding to the sequence of odd roots
\begin{align*}
\delta_n-\ep_1,\delta_n-\ep_2,\ldots,\delta_n-\ep_m.
\end{align*}
{The resulting simple system contains the simple odd non-isotropic root $\delta_n$ and the main task is then to find necessary and sufficient conditions for $\la$ so that $L^\fb(\la)$ is $f_{2\delta_n}$-free. This turns out to be the same as $(\la^{[m]},\delta_n)<0$}. \end{proof}

The following is an equivalent formulation of Theorem \ref{thm:injpro:d}.
\begin{cor}
The projective cover $P(\la)$ is injective if and only if $(\la+\rho,\alpha^\vee)\le 0$, for all $\alpha\in\Pi(\fn_{\oa}^+)$.
\end{cor}

Below we shall classify injective projective covers for the three exceptional simple Lie superalgebras $D(2|1,\zeta)$, $G(3)$, and $F(3|1)$. As proofs are similar, we shall omit them. Also in order to save space we shall use the notation from \cite{CSW18} without further explanation.

\subsubsection{The case $D(2|1,\zeta)$}

Let $\g=D(2|1,\zeta)$ and consider the triangular decomposition determined by the following simple system:
\begin{align*}
\Pi=\{\delta-\ep_1-\ep_2,2\ep_1,2\ep_2\}.
\end{align*}

\begin{thm}\label{thm:injpro:d21}
Let $\lambda=\la_1\delta+\mu_1\ep_1+\mu_2\ep_2\in\h^*$. Then the projective cover $P(\la)$ is injective if and only if $\lambda_1\le 1$ and $\mu_1<0$, $\mu_2<0$. Furthermore, if $\la_1=1$, then in addition $(1+\mu_1)\not=\pm(1+\mu_2)\zeta$.
\end{thm}

Equivalently we have:
\begin{cor}
The projective cover $P(\la)$ is injective if and only if $(\la+\rho,\alpha^\vee)\le 0$, for all $\alpha\in\Pi(\fn_{\oa}^+)$. {In the case when $(\la+\rho,\delta)=0$, we require additionally that $\la$ is a typical weight.}
\end{cor}

\subsubsection{The case $G(3)$}

Let $\g=G(3)$ and consider the triangular decomposition determined by the following simple system:
\begin{align*}
\Pi=\{\ep_2-\ep_1,\ep_1,\delta-\ep_1-\ep_2\}.
\end{align*}

\begin{thm}\label{thm:injpro:g3}
Let $\lambda=\la_1\delta+\mu_1\ep_1+\mu_2\ep_2\in\h^*$. Then the projective cover $P(\la)$ is injective if and only if $\lambda_1\le 2$ and $2\mu_1<\mu_2<\mu_1$.
\end{thm}

Equivalently we have:
\begin{cor}
The projective cover $P(\la)$ is injective if and only if $(\la+\rho,\alpha^\vee)\le 0$, for all $\alpha\in\Pi(\fn_{\oa}^+)$.
\end{cor}

\begin{rem} Characters of tilting modules in $\mathcal O$ were computed for the exceptional simple Lie superalgebras $D(2|1,\zeta)$ and $G(3)$ in \cite{ChWa17} and \cite{ChWa18}, respectively. Combining with Ringel duality this gives a classification of the projective tilting modules in  loc.~cit.~for these two Lie superalgebras.  Theorems \ref{thm:injpro:d21} and \ref{thm:injpro:g3} above confirm this classification.
\end{rem}

\subsubsection{The case $F(3|1)$}

Let $\g=F(3|1)$ and consider the triangular decomposition determined by the following simple system:
\begin{align*}
\Pi=\{\epsilon_1-\epsilon_2,\epsilon_2-\epsilon_3,\epsilon_3,1/2(\delta-\epsilon_1-\epsilon_2-\epsilon_3)\}.
\end{align*}

\begin{thm}
Consider the integral highest weight $\lambda=\lambda_1\delta+\sum_{i=1}^3\mu_i\epsilon_i$ with $\lambda_1,\mu_i$ integers or half-integers. Then the projective cover $P(\lambda)$ is injective if and only of $\lambda_1\le 3/2$ and $\mu_1<\mu_2<\mu_3\le -1/2$. If $\lambda=3/2$, then we need in addition $\mu_1+1/2-\mu_2-\mu_3\not=0$.
\end{thm}

We have equivalently:
\begin{cor}
The projective cover $P(\la)$ is injective if and only if $(\la+\rho,\alpha^\vee)\le 0$, for all $\alpha\in\Pi(\fn_{\oa}^+)$. Furthermore, if $(\la+\rho,\delta)=0$, then $\la$ is a typical weight.
\end{cor}

\section{Parabolic subalgebras of the periplectic Lie superalgebra} \label{Sect::newparaboDec}
\subsection{Periplectic Lie superalgebra}\label{Intropen}
For positive integers $m,n\geq 1$, the general linear Lie superalgebra $\mathfrak{gl}(m|n)$ may be realised as the space of $(m+n) \times (m+n)$ complex matrices
\begin{align*}
\left( \begin{array}{cc} A & B\\
C & D\\
\end{array} \right),
\end{align*}
where $A,B,C$ and $D$ are respectively $m\times m, m\times n, n\times m, n\times n$ matrices, with Lie bracket given by the super commutator.
Let $E_{ab}$ be the elementary matrix in $\mathfrak{gl}(m|n)$ with $(a,b)$-entry $1$ and other entries 0, for $1\leq a,b \leq m+n$.

The standard matrix realisation of the periplectic Lie superalgebra $\pn$ inside $\mathfrak{gl}(n|n)$ is given
by
\begin{align}\label{plrealization}
 \mf g:= \mf{pe}(n)=
\left\{ \left( \begin{array}{cc} A & B\\
C & -A^t\\
\end{array} \right)\| ~ A,B,C\in \C^{n\times n},~\text{$B^t=B$ and $C^t=-C$} \right\}.
\end{align}
Throughout this section, we fix the Cartan subalgebra $\mf h= \mf h_\oa \subset \mf g_\oo$ consisting of diagonal matrices. We denote the dual basis of $\mf h^*$ by $\{\vare_1, \vare_2, \ldots, \vare_n\}$ with respect to the standard basis
$$\{E_{ii}-E_{n+i,n+i}|~1\leq i \leq n \}\subset \pn.$$
All Borel and parabolic subalgebras we consider below are with respect to this Cartan subalgebra. The set of roots is given by
\begin{equation}\label{eqroots}
\Phi\;=\{\epsilon_i-\epsilon_j\,|\, 1\le i\not=j\le n\} \amalg \{ \epsilon_i+\epsilon_j\,|\, 1\le i\le j\le n\} \amalg  \{- \epsilon_i-\epsilon_j\,|\, 1\le i< j\le n\}.
\end{equation}
The Weyl group $W=\mf S_n$ of $\pn$ is the symmetric group on $n$ symbols. By Section \ref{Sect::claLiesup} we can  fix a   Borel subalgebra of $\mathfrak{gl}(n)=\mathfrak{pe}(n)_{\oa}$, which we choose to be the subalgebra consisting of matrices \eqref{plrealization} above with $B=C=0$ and $A$ upper triangular, and which we denote by $\fb_{\oa}^s$. Unless mentioned otherwise, all Borel and parabolic subalgebras (excluding their negative Borel subalgebras) are assumed to contain $\fb_{\oa}^s$.

We define the following subalgebras of $\pn$:
\begin{align*}
\mf g^+:=
\{\begin{pmatrix}
0 & B \\
0 & 0
\end{pmatrix}|B^t=B\}\quad\mbox{and}\quad \mf g^-:=
\{\begin{pmatrix}
0 & 0 \\
C & 0
\end{pmatrix}|C^t=-C\}.
\end{align*}


We normalise the non-degenerate $\mf S_n$-invariant bilinear form $(\cdot, \cdot): \mf h^*\times \mf h^* \rightarrow \C$ by $(\vare_i, \vare_j) =\delta_{ij}$, for all $1\leq i,j\leq n$. Finally, we let $X=\sum_{i=1}^n\mZ\ep_i$,  and  $\omega_k:=\vare_1+\cdots +\vare_k$, for any $1\leq k\leq n$. In particular, we have $\omega_n=\rho(\fg)$.

\vskip 0.5cm


\subsection{Some combinatorial definitions} \label{Sect::comdef}
For a partition $\lambda$ we denote by $\ell(\lambda)$ the length of $\lambda$. A bipartition is a pair $(\lambda|\mu)$ of two partitions $\lambda$ and $\mu$. For a bipartition denote by $\lambda\mu$ the composition $(\lambda_1,\lambda_2,\ldots,\lambda_{\ell(\lambda)}, \mu_1,\mu_2,\ldots,\mu_{\ell(\mu)}).$ There is a unique partition associated to $\lambda\mu$ obtained by reordering the parts appropriately which we denote by $\lambda\ast\mu$.

 We denote by $\RP$ the set of $2$-restricted partitions. These are all sequences $\lambda=(\lambda_1,\lambda_2,\lambda_3,\ldots)$ with $\lambda_i\in\mN$ and $0\le \lambda_i-\lambda_{i+1}\le 1$ and $\lambda_r=0$ for some $r\ge 0$.
For $r\in\mN$, we denote by $\RP^0_r\subset \RP$ and $\RP_r\subset \RP$ the sets of $2$-restricted partitions of length exactly $r$ and of length $\leq r$, respectively. For example, we have
$$\RP^0_2\;=\;\{\text{\tiny\Yvcentermath1$\yng(1,1)$} , \text{\tiny\Yvcentermath1$\yng(2,1)$} \}\;\mbox{ and }\;\RP^0_3\;=\;\{\text{\tiny\Yvcentermath1$\yng(1,1,1)$} , \text{\tiny\Yvcentermath1$\yng(2,1,1)$}, \text{\tiny\Yvcentermath1$\yng(2,2,1)$}, \text{\tiny\Yvcentermath1$\yng(3,2,1)$} \}.$$ Also, we define $\partial^n:=(n,n-1,\ldots, 1)\in \RP_n^0$.

\begin{df}
	The set $\BRP_n$ comprises all bipartitions $(\mu,\nu)$ such that the partition $\mu\ast \nu$ is in $\RP_n$. The set $\BRP_n^0\subset\BRP_n$ comprises all bipartitions $(\mu,\nu)$ such that $\mu_i=\nu_j$ implies $\mu_i=\nu_j=0$ and $\mu\ast\nu$ is in $\RP_n^0$. The set $\BRP_n^{00}\subset\BRP_n^0$ comprises all bipartitions $(\mu,\nu)$ for which $\mu\ast\nu=\partial^n$.
\end{df}

As an example, we have
$$\BRP_2^0\;=\;\BRP_2^{00}\amalg\{(\text{\tiny\Yvcentermath1$\yng(1,1)$},\varnothing),(\varnothing,\text{\tiny\Yvcentermath1$\yng(1,1)$})\}\quad\mbox{and}\quad \BRP^{00}_2\;=\; \{(\text{\tiny\Yvcentermath1$\yng(2,1)$},\varnothing),(\text{\tiny\Yvcentermath1$\yng(2)$},\text{\tiny\Yvcentermath1$\yng(1)$}),(\text{\tiny\Yvcentermath1$\yng(1)$},\text{\tiny\Yvcentermath1$\yng(2)$}),(\varnothing,\text{\tiny\Yvcentermath1$\yng(2,1)$})\}.$$

The following lemma is a straightforward observation. We identify $\{+,-\}^{\times n}$ with functions from $\{1,2,\ldots,n\}$ to $\{+,-\}$ in an obvious way.

\begin{lem}\label{LemRPBij}
	We have a bijection
	$$\BRP^0_n\stackrel{\sim}{\to} \{(\kappa,f)\in \RP_n^0\times \{+,-\}^{\times n}\,|\, f(i)=f(j)\mbox{ whenever $\kappa_i=\kappa_j$}\},$$
	given by $(\mu,\nu)\mapsto (\mu\ast\nu,f)$ with $f(i)=+$ if and only if $(\mu\ast\nu)_i$ appears in $\mu$. This restricts to a bijection
	$$\BRP^{00}_n\stackrel{\sim}{\to} \{\partial^n\}\times \{+,-\}^{\times n}\cong\{+,-\}^{\times n}.$$
\end{lem}
In particular, we find $|\BRP^{00}_n|=2^n$.

\begin{rem}\label{rem:BRPid} In light of Lemma \ref {LemRPBij} we shall identify $\BRP^0_n$ with the subset of $\{(\kappa,f)\in \RP_n^0\times \{+,-\}^{\times n}$ satisfying $f(i)=f(j)$ when $\kappa_i=\kappa_j$.
\end{rem}

\subsection{Classification}
Since $\fh$ is equipped with the non-degenerate bilinear form $(\cdot,\cdot)$, we can replace $H\in\fh$ in the defining equation \eqref{deflu} of parabolic decompositions by an element $\delta\in\fh^\ast$ and let the decompositions be determined by $\Real(\delta,\alpha)$ for $\alpha\in\Phi$. We then use $\mf l(\delta), \mf u^\pm(\delta), \mf n^\pm(\delta), \mf b(\delta)$ and $\mathfrak{p}(\delta)$ to denote the corresponding subalgebras and have a surjective map $\delta\mapsto (\fp(\delta),\fl(\delta))$ from $\fh^\ast$ to the set of parabolic decompositions of $\fg$.

We now define a map
\begin{equation}\label{Defzetax}\BRP_n\;\hookrightarrow\;\fh^\ast,\;\; x=(\mu,\nu)\,\mapsto\, \zeta_x= \sum_{i\ge 1}\mu_i\epsilon_i -\left(\sum_{j\ge 1}\nu_j\epsilon_{n+1-j}\right).\end{equation}
Injectivity of this map follows from $\ell(\mu)+\ell(\nu)\le n$. In what follows we will make no distinction between $x\in\BRP_n$ and $\zeta_x\in\fh^\ast$. 

Recall that we only consider parabolic subalgebras containing the Borel subalgebra $\fb_{\oa}^s$ of $\mathfrak{gl}(n)$ and that every other parabolic subalgebra is conjugate to one of this form.

\begin{thm} \label{ClaParaSubal}
	Denote the composition of $\BRP_n\hookrightarrow \fh^\ast$ and the map from $\fh^\ast$ to the set of parabolic subalgebras of $\pn$ by $\phi: x\mapsto( \fp(x),\mf l(x))$. 
	\begin{enumerate}[(i)]
		\item The map $\phi$ yields a bijection between $\BRP_n$ and the set $\{ (\mf p, \mf l)|~\mf p \supset \fb_{\oa}^s\}$ of all parabolic decompositions with parabolic subalgebra containing $\fb_{\oa}^s$.  \label{ClaParaSubalEq1}
		\item The map $\phi$ yields a bijection between $\BRP_n^0$ and the set of all reduced parabolic subalgebras containing $\fb_{\oa}^s$. \label{ClaParaSubalEq2}
		\item The map $\phi$ yields a bijection between $\BRP_n^{00}$ and $\Bor(\fg,\fb_{\oa}^s)$. \label{ClaParaSubalEq3}
	\end{enumerate}
\end{thm}

\begin{rem}
	Part (iii) of Theorem~\ref{ClaParaSubal} can be found in \cite[\S 3.6.2]{Mu12}.
\end{rem}

\begin{proof} We first show that $\phi$ in  \eqref{ClaParaSubalEq1} is surjective.

For a subring $R\subset \mC$ we denote by $I_R\subset\fh^\ast$ the $R$-span of the elements $\{\epsilon_i\,|\,1\le i\le n\}$. Since $\Phi\subset I_{\mR}$, it follows that $(\mf p,\mf l)$ only depends on real part of $\delta$, so
	$$\{(\mf p(\delta),\fl(\delta))|~ \delta \in \mf h^\ast \} = \{ (\mf p(\delta),\fl(\delta))|~\delta \in I_{\mR} \}.$$

	For $r\in\mR$, by {\em the sign of $r$} we refer to whether we have $r>0$, $r=0$ or $r<0$.
	By Equation~\eqref{eqroots}, we find that $(\fp(\delta),\fl(\delta))$ for $\delta=\sum_i\delta_i\epsilon_i$ in $I_{\mR}$ depends only on the signs of $\delta_i-\delta_j$ and $\delta_i+\delta_j$. We will freely use this.
	
	For an arbitrary $\delta\in I_{\mR}$, we claim that there exists $\gamma\in I_{\mZ}$ which induces the same parabolic decomposition as $\delta$. Firstly we observe that $\delta\mapsto a\delta$ does not affect the decomposition, for $a\in \mR_{>0}$. We can thus assume that all values $|\delta_i\pm \delta_j|$ are not in the open interval $(0,2)$. For such a $\delta$, replacing $\delta$ by $\sum_i \delta_i'\epsilon_i$ with
	$$\delta_i'=\begin{cases} \lceil \delta_i\rceil&\mbox{if $\delta_i>0$}\\
	0&\mbox{if $\delta_i=0$}\\
	\lfloor \delta_i\rfloor&\mbox{if $\delta_i<0$,}
	\end{cases}$$  will not change the signs of $\delta_i-\delta_j$ and $\delta_i+\delta_j$.
	In conclusion, we have
	$$\{(\mf p(\delta),\fl(\delta))|~ \delta \in \mf h^\ast \} = \{ (\mf p(\delta),\fl(\delta))|~\delta \in I_{\mZ} \}.$$
	Next,   set $I^+_{\mZ}:= \{\delta \in I_{\mZ}|~\delta_i\geq \delta_{i+1}, \text{ for all }1\leq i\leq n-1 \}$. It is clear that we have $\fp(\delta)$  for $\delta\in I_{\mZ}$ will contain $\fb_{\oa}^s$ if and only if $\delta \in I^+_{\mZ}.$


	
	For any $\delta \in I^+_{\mZ}$ there exists $1\le p\le n$ such that $\delta= \sum_{i=1}^{p}\mu_i\epsilon_i+{\sum_{j=1}^{n-p} v_j\epsilon_{p+j}}$ with $$\mu_1\geq \mu_2\geq \cdots \geq \mu_p \geq 0 > v_1\geq v_2\geq \cdots \geq v_{n-p}.$$
	We describe three manipulations of $\delta\in I^+_\mZ$ which do not affect $(\fp(\delta),\fl(\delta))$, since they do not affect the signs of $\delta_i+\delta_j$ and $\delta_i-\delta_j$. We will use the convention $\mu_0=+\infty$ and $v_{n-p+1}=-\infty$.
	
	(1) Assume that there are $ i,j,s$  such that $$\mu_{i-1}>\mu_{i}=\mu_{i+1}=\ldots =\mu_{i+s}>\mu_{i+s+1}+1,$$
	$$v_j-1>-\mu_{i}>v_{j+1},$$   then $\delta-\vare_{i}-\vare_{i+1}-\cdots -\vare_{i+s}$ and  $\delta$ yield the same parabolic decomposition.

	(2) Assume that there are $ i,j,t$ such that $$v_{j-1}-1>v_{j}=\ldots =v_{j+t}>v_{j+t+1},$$
	$${\mu_{i}>-v_{j}>\mu_{i+1}}+1,$$
	then $\delta+\vare_{p+j}+\vare_{p+j+1}+\cdots +\vare_{p+j+t}$ and  $\delta$ yield the same parabolic decomposition.

	(3) Assume that there are $i,j$ such that $\mu_i= -v_j$ and  $$\mu_{i-1}>\mu_{i}=\mu_{i+1}=\ldots =\mu_{i+s}>\mu_{i+s+1}+1,$$
	$$v_{j-1}-1>v_{j}=\ldots =v_{j+t}>v_{j+t+1},$$  then $\delta-\vare_{i}-\vare_{i+1}-\cdots -\vare_{i+s}+\vare_{p+j}+\vare_{p+j+1}+\cdots +\vare_{p+j+t}$ and  $\delta$ yield the same parabolic decomposition.

	It is a straightforward observation that the image of $\BRP_n\hookrightarrow \fh^\ast$ consists precisely of those elements of $I^+_{\mZ}$ for which none of the conditions (1)--(3) hold. Using the manipulations (1)--(3) repeatedly we eventually arrive at a $\delta'\in I^+_{\mZ}$ with $(\mf p(\delta),\fl(\delta)) =(\mf p(\delta'),\fl(\delta'))$ such that none of conditions in (1)--(3) hold for $\delta'$. Therefore  $\delta'$ lies in $\BRP_n$, which shows surjectivity of $\phi$ in  \eqref{ClaParaSubalEq1}.
	
	We now show that $\phi$ in   \eqref{ClaParaSubalEq1} is injective.  Assume that $(\fp(\delta),\fl(\delta)) =(\fp(\delta'),\fl(\delta'))$, for $\delta,\delta'\in\BRP_n$. Let us write
	$$\delta= \sum_{i=1}^{p}\mu_i\epsilon_i+\sum_{j=p+1}^n v_j\epsilon_{j},~\delta' =  \sum_{i=1}^{p'}\mu'_i\epsilon_i+\sum_{j=p'+1}^n v'_j\epsilon_{j},$$ where
	$(\mu_1, \mu_2, \ldots \mu_p|-v_{n}, -v_{n-1},\ldots,-v_{p+1})$ and $(\mu'_1, \mu'_2, \ldots \mu'_{p'}|-v'_{n}, -v'_{n-1},\ldots,-v'_{p'+1})$,  are respectively  associated bi-partitions for $\delta$ and $\delta'$ defined in Section \ref{Sect::comdef}. Here we ignore the zeros in the above expressions.

	For
given $i\neq j$ and $k$,  we have the following facts
	\begin{enumerate}
		\item[(a)]  $\mf l(\delta) \supset \mf g_{\vare_i-\vare_j} $ if and only if $\delta_i=\delta_j$. \label{ClaParaSubalEq_1}
		\item[(b)] $\mf l(\delta) \supset \mf g_{\vare_i+\vare_j}$ if and only if $\delta_i = -\delta_j$. \label{ClaParaSubalEq_2}
		\item[(c)] $\mf l(\delta) \supset \mf g_{2\vare_k}$ if and only if $\delta_k =0$. \label{ClaParaSubalEq_3}
		\item[(d)]  $\mf p(\delta) \supset \mf g_{\vare_i+\vare_j}$ if and only if $\delta_i\ge -\delta_j$.
		\item[(e)] $\fp(\delta)\supset \fg_{2\epsilon_k}$ if and only if $\delta_k\ge 0$.
	\end{enumerate}
	
	By $(a)$--$(e)$, and the analogous claims for $\delta'$, it follows from $(\fp(\delta),\mf l(\delta)) =(\fp(\delta'),\mf l(\delta'))$ that $p=p'$ and for any $i,j$ we have
	\begin{itemize}
		\item  $\mu_i=\mu_j$ if and only if $\mu'_i=\mu'_j$.
		\item  $v_i=v_j$ if and only if $v'_i=v'_j$.
		\item $\mu_i= -v_j$ if and only if $\mu'_i= -v'_j$.
		\item $\mu_i=0$ if and only if $\mu'_i=0$.
		\item $\mu_i> -v_j$ if and only if $\mu'_i> -v'_j$.
	\end{itemize}
		It follows from these observations and the fact that $\delta,\delta'$ are in $\BRP_n$ that $\delta =\delta'$. This proves that $\phi$ in the claim \eqref{ClaParaSubalEq1} is a bijection.
	
	We now show \eqref{ClaParaSubalEq2}. For $\delta\in \BRP_n$, observations  $(b)$ and $(c)$ show that
	\begin{align*}
	&\mf l(\delta) \subset \mf g_\oo\; \Leftrightarrow\; \delta_k\neq 0\text{ and }\delta_k+ \delta_\ell\neq 0,  \text{ for any }1\leq k, \ell\leq n \;\Leftrightarrow\; \delta \in \BRP_n^0.
	\end{align*}
	This proves bijectivity of $\phi$ in claim \eqref{ClaParaSubalEq2}.
	
	Finally, For $\delta\in \BRP_n$, observations  $(a)$, $(b)$ and $(c)$ show that
	\begin{align*}
	&\mf l(\delta) = \mf h \Leftrightarrow \delta_k\neq 0\text{ and }  \delta_k\pm\delta_\ell\neq 0,  \text{ for any }1\leq k\neq \ell\leq n \Leftrightarrow \delta \in \BRP_n^{00}.
	\end{align*}
	This proves bijectivity of $\phi$ in claim \eqref{ClaParaSubalEq3}.
\end{proof}

\begin{rem}
Let $\delta= \sum_{i=1}^{p}\mu_i\epsilon_i+{\sum_{j=1}^{n-p} v_j\epsilon_{p+j}} \in \fh^*$ correspond to an element in $\BRP_n$. For each $1\le i\le n$ we define
\begin{align*}
k_i&:=|\{\mu_j|\mu_j=i\}|,\quad
\ell_i:=|\{v_j|v_j=-i\}|.
\end{align*}
Furthermore, set $k_0:=n-\sum_{i=1}^n(k_i+\ell_i)$. Then, one verifies that
\begin{align*}
\fl(\delta)\cong \mf{pe}(k_0) \oplus\bigoplus_{i=1}^n\mf{gl}(k_i|\ell_i).
\end{align*}
\end{rem}

\begin{ex}\label{examp2}
	For $\mf g=\mf{pe}(2)$, the elements $(\text{\tiny\Yvcentermath1$\yng(2,1)$}, \varnothing)$   and $(\text{\tiny\Yvcentermath1$\yng(1)$},\varnothing)$ of $\BRP_2$ respectively correspond to the
	parabolic decompositions
	\begin{align*}
	&\mf u^- = \mf g^{\vare_2-\vare_1} \oplus \fg^{-\vare_1-\vare_2},\; \mf l = \mf h,\; ~ \mf u^+ =  \mf g^{\vare_1-\vare_2} \oplus  \bigoplus_{1\leq i\leq  j\leq 2}\mf g^{\vare_i+\vare_j},\\
	&\mf u^- = \mf g^{\vare_2-\vare_1} \oplus \fg^{-\vare_1-\vare_2},\; \mf l = \mf h\oplus \mf g^{2\vare_2},\; ~ \mf u^+ =  \mf g^{\vare_1 -\vare_2} \oplus \mf g^{2\vare_1} \oplus \mf g^{\vare_1+\vare_2}.
	\end{align*}
	Both decompositions have as parabolic subalgebra the standard Borel subalgebra $\fb^s$.
\end{ex}

\begin{ex}
	(1). The element $((1)^n,\varnothing)\in\BRP_n^0$ gives rise to the maximal parabolic subalgebra
	\begin{align*}
	\mf p=
	\{\begin{pmatrix}
	A & B \\
	0 & -A^t
	\end{pmatrix}|B^t=B\},~\text{ with } \mf l=\mf g_\oo \quad\mbox{and}\quad \fu^+=\fg^+.
	\end{align*}
	(2). The element  $(\varnothing,(1)^n)\in\BRP_n^0$ gives rise to the maximal parabolic subalgebra
	\begin{align*}
	\mf p=
	\{\begin{pmatrix}
	A & 0 \\
	C & -A^t
	\end{pmatrix}|C^t=-C\}, ~\text{ with } \mf l=\mf g_\oo\quad\mbox{and}\quad\fu^+=\fg^-.
	\end{align*}
\end{ex}

\subsection{Description of Borel subalgebras}
We describe the set $\Phi(\fb)$ for all Borel subalgebras in the classification of Theorem~\ref{ClaParaSubal}(iii).
\begin{prop}  \label{lem::DescriptionBorel}
	The Borel subalgebra $\fb(\zeta_x)$, for bipartition $x=(\mu,\nu)\in\BRP^{00}_n$ with $p=\ell(\mu)=n-\ell(\nu)$, has as odd roots
	$$\amalg_{1\le i\le p}\{\epsilon_i+\epsilon_j\,|\, i\le j\le \mu_i +i-1\}\;\amalg\; \amalg_{1\le k\le n-p }\{-\epsilon_{n+1-l}-\epsilon_{n+1-k}\,|\, k<l\le \nu_k+k-1 \}.$$
\end{prop}
\begin{proof}
	This follows from the explicit form of $\zeta_x$ in \eqref{Defzetax}.
\end{proof}

\begin{cor}\label{CorDim}
	For bipartition $x=(\mu,\nu)\in\BRP^{00}_n$ with $p=\ell(\mu)$, we have
	$$\dim_{\mC}\fb(\zeta_x)_{\ob}\;=\;\frac{1}{2}n(n-1)+p.$$
\end{cor}

\begin{cor}\label{CorDim2} Let $x=(\mu,\nu)$ and $ x'=(\mu',\nu')$ be bipartitions in $\BRP^{00}_n$.
	Then $\mf b(\zeta_x) \subsetneq \mf b(\zeta_{x'})$ if and only if $\mu' = (\mu, 1)$ and $\nu= (\nu',1)$.
\end{cor}
\begin{proof}
	Let  $p=\ell(\mu)$ and  $p'=\ell(\mu')$. In this proof, we will freely use the description of roots of Borel subalgebras as given in Proposition \ref{lem::DescriptionBorel}.
It follows from Corollary \ref{CorDim} that	$p'>p$. Therefore the proof is divided into the following two cases:
	
	

	{\bf Case 1}: Assume that  $ p' =p+1$.  In this case we have $\mu_i\leq \mu_i'$ and $\nu_j\leq \nu_j'$, for all $1\leq i \leq p$, $1 \leq j \leq n-p-1$. We claim that $\nu_{n-p}=1$. Suppose on the contrary that $\nu_{n-p}>1$ then $-\vare_p-\vare_{p+1}$ is a root of $\mf b$.
	But $p'=p+1$ implies that $\vare_p+\vare_{p+1}\in\fb'$ and hence $-\vare_p-\vare_{p+1}\not\in\fb'$, a contradiction. 	Consequently, in this case we have $\mu' = (\mu, 1)$ and $\nu =(\nu', 1)$.
	
	{\bf Case 2}: Assume that $p' >p+1$. In this case $\mu_{p+1}'>1$ so $\vare_{p+1}+\vare_{p+2}$ is a root of $\mf b(\zeta_{x'})$, {and in particular $-\vare_{p+1}-\vare_{p+2}\not\in\mf b(\zeta_{x'})$. But $-\ep_{p+1}-\ep_{p+2}\in\fb(\zeta_x)$, a contradiction.}
	
	This completes the proof.
\end{proof}

\begin{ex} Recall the standard realisation of $\mathfrak{pe}(n)$ in Section~\ref{Intropen}.
	\begin{enumerate}[(i)]
		\item  The {\it standard Borel subalgebra} $\frakb^s$ comprises all matrices of the form \eqref{plrealization} with $A$ upper triangular, $B$ symmetric and $C=0$. This Borel subalgebra corresponds to $(\partial^n,\varnothing)\in\BRP_n^{00}$. We have
		$$\Phi(\fb^s)= \{\vare_i-\vare_j,\vare_i+\vare_j, 2\epsilon_i|i< j\}.$$
		\item  The {\it reverse Borel subalgebra}  $\mathfrak{b}^r$ comprises all matrices of the form \eqref{plrealization} with $A$ upper triangular, $C$ skew-symmetric, and $B=0$.  This Borel subalgebra corresponds to $(\varnothing,\partial^n)\in\BRP_n^{00}$. We have
		$$\Phi(\fb^r)=  \{\vare_i-\vare_j,-\vare_i-\vare_j|i< j\}.$$
	\end{enumerate}
\end{ex}

By Corollary~\ref{CorDim}, the above Borel subalgebras are uniquely determined by
$$\dim_{\mC}\fb^s_{\ob}\;=\;\frac{1}{2}n(n+1)\qquad\mbox{and}\qquad \dim_{\mC}\fb^r_{\ob}\;=\;\frac{1}{2}n(n-1).$$


\subsection{Description of reduced parabolic subalgebras}
\begin{df} \label{df::incl}
	{We define a partial order $\le$ on $\BRP^0_n$, which we identify with a subset of $\RP_n^0\times\{+,-\}^{\times n}$ as in Remark \ref{rem:BRPid}, by $(\kappa,f)\le (\kappa',f')$ if}
	\begin{enumerate}
		\item $(\kappa_1-\kappa_1',\kappa_2-\kappa_2',\ldots)$ is a partition, and \label{df::incl1}
		\item $f(i)=f'(i)$ for $i<n$ and $(f(n),f'(n))\in\{(-,+),(+,+),(-.-)\}$.\label{df::incl2}
	\end{enumerate}
\end{df}



By Lemma~\ref{LemRPBij}, we can label reduced parabolic subalgebras by pairs $(\kappa,f)$ and Borel subalgebras by $f$. We use this notation in the following proposition, which makes the general observation in Lemma~\ref{Lembp} concrete for $\mathfrak{pe}(n)$.
\begin{prop} \label{prop::paraborel}  Let $(\kappa, f), (\kappa', f)$ be in the set in the right-hand side of Lemma  \ref{LemRPBij}. Then we have $\mf p(\kappa, f)_{\ob}=\mf p(\kappa', f)_{\ob}$. In particular,   $\mf p(\kappa, f)_\ob = \mf b(f)_\ob$.
\end{prop}
 \begin{proof}Let $x= (\mu,\nu)$ and $x'=(\mu',\nu')$ be the  bi-partitions associated with $(\kappa,f)$ and $(\kappa',f)$,  respectively. Namely, $\mu\ast\nu = \kappa$ and $\mu' \ast\nu' = \kappa'$.
		
	Let $1\leq p\leq n$.  We first note that $2\vare_p$ is a root of $\fp(\kappa, f)$ if and only if $\ell(\mu)\ge p$, which is equivalent to $\ell(\mu')\ge p$. Therefore $2\vare_p$ is a root of $\fp(\kappa, f)$ if and only if $2\vare_p$ is a root of $\fp(\kappa', f)$. In this case, $\vare_p+\vare_q$ are roots of both $\fp(\kappa, f)$ and $\fp(\kappa', f)$ for all $q=p,\ldots,\ell(\mu)=\ell(\mu')$.  Also, since every parabolic subalgebra is in particular a $\fb_{\oa}$-submodule of $\fg$, for given $1\leq s\leq t$ we note that if $\vare_s+\vare_t$ is not a root of $\fp(\kappa, f)$  (resp. $\fp(\kappa', f)$) then $\vare_s+\vare_{t'}$ is also not a root of $\fp(\kappa, f)$  (resp. $\fp(\kappa', f)$) for any $t\leq t'$.
	
	Assume that $2\vare_p$ is a root of  $\fp(\kappa, f)$ and set  $q=\ell(\mu)+j$, for $1\leq j \leq n-\ell(\mu)$. Then  $\vare_p+\vare_q$ is not a root of $\fp(\kappa, f)$ if and only if $\mu_p <\nu_{n-q+1}$. Furthermore, we have
{	$$	 |\{q|~\text{$\vare_p+\vare_q\not\in\Phi(\fp(\kappa, f))$}\}|=|\{j|~\text{$\mu_p<\nu_{j}$}\}|.$$}
Since $f=f'$ we have
$|\{j|~\text{$\mu_p<\nu_{j}$}\}|=|\{j|~\text{$\mu'_p<\nu'_{j}$}\}|$, which gives that
\begin{align*}
|\{q|~\text{$\vare_p+\vare_q\not\in\Phi(\fp(\kappa, f))$}\}|=|\{q|~\text{$\vare_p+\vare_q\not\in\Phi(\fp(\kappa', f))$}\}|.
\end{align*}
	 This means that $\mf p(\kappa, f)_\ob\cap \mf g^+ = \mf p(\kappa', f)_\ob\cap \mf g^+$. Using an analogous argument we have $\mf p(\kappa, f)_\ob\cap \mf g^- = \mf p(\kappa', f)_\ob\cap \mf g^-$.   As a consequence, we have   $\mf p(\kappa, f)_{\ob}=\mf p(\kappa', f)_{\ob} = \mf b(f)_\ob$.
\end{proof}

\begin{prop} Let $(\kappa, f), (\kappa', f')$ be in the set in the right-hand side of Lemma  \ref{LemRPBij}. We have
	$\mf p(\kappa, f) \subseteq \mf p(\kappa', f')$ if and only if $(\kappa,f)\le (\kappa',f')$.
\end{prop}
	\begin{proof}
	We consider  $(\kappa, f), (\kappa', f')$ and the corresponding elements in $\BRP^0_n$ which we regard as $\delta,\delta'\in\fh^\ast$, respectively.
	We let $p$ be the number of times $f'(i)=+$, for $1\le i\le n$, which equals the number of positive labels in $\delta'=\sum_i\delta'_i\epsilon_i$.
	
	By Proposition~\ref{prop::paraborel} and Corollary~\ref{CorDim2}, we have $\fp(\kappa,f)_{\ob} \subseteq\fp(\kappa',f')_{\ob}$ if and only if either $f=f'$ or $f(i)=f'(i)$ for $i<n$ and $f(n)=-$, $f'(n)=+$. On the other hand, it follows immediately that
	\begin{equation}\label{eqdelta}\fp(\delta)_{\oa} \subseteq\fp(\delta')_{\oa}\quad\Leftrightarrow\quad {\fl(\delta) \subseteq\fl(\delta')}\quad\Leftrightarrow \quad\{\delta_{i}=\delta_j\text{ implies } \delta'_i=\delta'_j,\text{ for all $1\le i,j\le n$}\}.\end{equation}
	
	Based on the observation on the odd part of the parabolic subalgebras, we can divide into three cases.
	
	{\em (1) We have $f=f'$.} In this case it follows that the condition in the right-hand side of \eqref{eqdelta} is equivalent to the condition that $\kappa_k=\kappa_l$ implies $\kappa'_k=\kappa'_l$ for all $1\le k,l\le n$. The latter is equivalent to Definition \ref{df::incl}(1). So under these assumptions, $\fp(\kappa,f) \subseteq\fp(\kappa',f')$ is equivalent to $(\kappa,f) \le(\kappa',f')$.
	
	{\em (2) We have $f(i)=f'(i)$ for $i<n$ and $f(n)=-$, $f'(n)=+$. Suppose that either $\delta'_p=\delta'_{p-1}$ or $\delta_p=\delta_{p+1}$. } We just deal with the case $\delta_p=\delta_{p+1}$, the other case being similar. By assumption, we have $\epsilon_p-\epsilon_{p+1}\in\Phi(\fl(\delta))$. However, since $\delta'_p>0$ and $\delta'_{p+1}<0$, we have $\epsilon_p-\epsilon_{p+1}\not\in\Phi(\fl(\delta'))$ and by \eqref{eqdelta} therefore $\fp(\delta)_{\oa}\not \subseteq\fp(\delta')_{\oa}$. Furthermore, since $f(n)=-$ we have $\delta_p=-\kappa_n$, and since $\delta_p=\delta_{p+1}$ we have $\kappa_n=\kappa_{n-1}$, so in particular $-=f(n-1)=f'(n-1)$. Consequently, by the condition in the right-hand side of Lemma  \ref{LemRPBij}, we have $\kappa_n'\not=\kappa'_{n-1}$ and it follows that $(\kappa,f)\not\le (\kappa',f')$. Hence we never have a relation in the partial order or an inclusion, under these assumptions.

	{\em (3) We have $f(i)=f'(i)$ for $i<n$ and $f(n)=-$, $f'(n)=+$ and furthermore both $\delta'_p\not=\delta'_{p-1}$ and $\delta_p\not=\delta_{p+1}$. } Since $\delta'_{p-1}>\delta'_p$ the pair $(\kappa',f)$ satisfies the condition on the right hand side of Lemma~\ref{LemRPBij} and thus we obtain an associated reduced parabolic subalgebra $\fp(\kappa',f)$. We claim that $\fp(\kappa',f)\subsetneq\fp(\kappa',f')$. Indeed, by construction, $\fp(\kappa',f)$ and $\fp(\kappa',f')$ differ only by the root vector corresponding to $2\ep_p$. Since the number of positive labels in $\delta$ equals $p-1$, this same root vector also cannot lie in $\fp(\kappa,f)$ and so we conclude that $\fp(\kappa,f)\subseteq\fp(\kappa',f')$ if and only if $\fp(\kappa,f)\subseteq\fp(\kappa',f)$. Now by Case (1) the latter is equivalent to Definition \ref{df::incl}(1).
	\end{proof}

\subsection{Category $\mc O$ of $\pn$ for arbitrary Borel subalgebras}


We define the duality function $(\cdot)^*:\BRP_n^0 \rightarrow\BRP_n^0$ by letting $(\mu,\nu)^\ast =(\nu,\mu)$, for all $(\mu,\nu)\in \BRP^0$. Then  we have the following   description.


\begin{prop} \label{lem::BvPlus}   For any $\delta\in \mc \BRP_n^0$ we have   $\hat{\mf p}(\delta) = \mf p(\delta^\ast).$
\end{prop}

\begin{ex} \label{ex::mainEx1}	
By  setting $\fp=\fb^s$  in Lemma \ref{Lem::dualtilting}, Theorem \ref{Thm::2ndmain} and Corollary \ref{Cor::ASdual}, we have the following:
\begin{align*}
&\bD L^{\mf b^s}(\mu) = L^{\mf b^r}({- w_0 \mu}),\quad \bD T^{\mf b^s}(\la)=T^{\mf b^r}(-w_0\la),\\
&\bD \LT P^{\mf b^s}(\la)= T^{\mf b^s}(-\la -2\rho_\oo-(1-n)\omega_n),\\
&(T^{\mf b^s}(\la): \Delta^{\mf b^s}(\mu))=[\nabla^{\mf b^s}(-\mu-2\rho_\oo):L^{\mf b^s}(-\la-2\rho_\oo)].
\end{align*}
\end{ex}

\subsubsection{Projective-injective modules} Let $\fb$ be an arbitrary Borel subalgebra. In this subsection we shall classify all injective $P^{\mf b}(\la)$. We first prove the following lemma.

\begin{lem}  \label{lem::precomputeVermaSocle}
	For any $\mu \in \h^*$ we have $[\Delta^{\mf b^s}(\mu):L^{\mf b}(\mu + 2\rho(\mf b\cap \mf g^{-}))]>0$. In particular, if
	 $\mu\in\fh^\ast$ is  anti-dominant, we have $L^{\mf b^s}(\mu) = \Delta^{\mf b^s}(\mu)  =L^{\mf b}(\mu + 2\rho(\mf b\cap \mf g^{-}))$.
\end{lem}
\begin{proof}
Choose a non-zero element $u \in S^{\text{top}}(\mf b\cap \mf g^-)$, which we regard as an element in $U(\fg)$. By construction, we have $Xu=0$ for $X\in\mf b\cap \mf g^-$. Since $\mf b\cap \mf g^-$ is an $\fn_{\oa}^+$-module, we have $[Y,u]=0$ for $Y\in\fn_{\oa}^+$. Finally, since $[\mf b\cap \mf g^+,\mf b\cap \mf g^-]\subset\fn_{\oa}^+$, we have $Zu\in U(\fg)\fn_{\oa}^+$ for $Z\in \mf b\cap \mf g^+$.

For a $\fb^s$-highest weight vector $v_\mu$ of $\Delta^{\mf b^s}(\mu)$, the non-zero vector $uv_\mu$ is by the previous paragraph a $\fb$-singular vector and has weight $\mu + 2\rho(\mf b\cap \mf g^{-})$. Hence $L^{\mf b}(\mu + 2\rho(\mf b\cap \mf g^{-}))$ must be a constituent of $\Delta^{\mf b^s}(\mu)$.

	
	
	If  $\la\in\fh^\ast$ is  anti-dominant, then by  \cite[Lemma 3.1]{Se02} we have $L^{\mf b^s}(\la) =\Delta^{\mf b^s}(\la)$. This completes the proof.	
\end{proof}

\begin{cor} \label{Lem::injantitilt} Let $\fb$ be an arbitrary Borel subalgebra. The following are equivalent:	\begin{itemize}
		\item[(1)]	$P^{\mf b}(\la)$ is a tilting module.
		\item[(2)] $P^{\mf b}(\la)=I^{\mf b}(\la +2\omega_n)$.
		\item[(3)]  $\la - 2\rho(\mf b\cap \mf g^{-})$ is anti-dominant.
		\item[(4)] {$P^{\mf b}(\la)= T^{\hat{\mf b}}(w_0\la -2\rho(\hat{\mf b}))$.}
	\end{itemize}

\end{cor}
\begin{proof} It follows from  Theorem \ref{prop::paraIrving} and Corollary \ref{Cor::ProTilt} that $(1),(2)$ and $(4)$ are equivalent. By  Corollary \ref{Cor::typeIPI}, $P^{\mathfrak{b}^s}(\mu)$ is injective if and only if $\mu$ is anti-dominant. For an arbitrary $\la$ let $\widetilde{\la}$ be such that $P^{\fb}(\la)\cong P^{\fb^s}(\widetilde{\la})$. Consequently, if $P^{\fb}(\la)\cong P^{\fb^s}(\widetilde{\la})$ is injective, then by Lemma~\ref{lem::precomputeVermaSocle}, $\la=\widetilde{\la}+ 2\rho(\mf b\cap \mf g^{-})$, so in particular $\la- 2\rho(\mf b\cap \mf g^{-})$ is anti-dominant.
If on the other hand, $\la - 2\rho(\mf b\cap \mf g^{-})$ is anti-dominant then by Lemma~\ref{lem::precomputeVermaSocle} $P^{\fb^s}(\la - 2\rho(\mf b\cap \mf g^{-}))\cong P^{\fb}(\la)$, so $P^{\fb}(\la)$ is injective. That (3) is equivalent to the other statements thus follows again from Theorem \ref{prop::paraIrving}.
\end{proof}

We classify all self-dual projective modules for the periplectic Lie superalgebra in the following corollary.
\begin{cor} \label{Cor::selfproj}
	Let $\la\in X$. Then the following are equivalent:
	\begin{enumerate}
		\item   $P^{\mf b^s}(\la)$ is self-dual for $\bD$.
		\item   $\la$ is anti-dominant with $\la + w_0\la =(n-3)\omega_n$.
	\end{enumerate}
\end{cor}
\begin{proof}
	We first note that $\bD P^{\mf b^s}(\la)\cong P^{\mf b^s}(\la)$ implies $P^{\mf b^s}(\la)$ is injective. Therefore by Corollary \ref{Lem::injantitilt} we may conclude that $P^{\mf b^s}(\la)$ is self-dual if and only if $\la$ is anti-dominant and  $\bD L^{\mf b^s}(\la+2\omega_n)\cong L^{\mf b^s}(\la)$.

	On the other hand, by Example~\ref{ex::mainEx1}, we have $\bD L^{\mf b^s}(\la+2\omega_n) = L^{\mf b^r}(-w_0\la -2\omega_n)$. Since $-w_0\la$ is also anti-dominant, we may conclude that
	$\bD L^{\mf b^s}(\la+2\omega_n) = L^{\mf b^s}(-w_0\la -2\omega_n -(1-n)\omega_n)$ by Lemma \ref{lem::precomputeVermaSocle}. This means that  $P^{\mf b^s}(\la)$ is self-dual if and only if $\la + w_0\la =(n-3)\omega_n$.
\end{proof}

\subsubsection{Tilting modules}



In subsection \ref{subsubsect::spo2n2m} we used odd reflections for  contragredient Lie superalgebras. We refer to \cite[Section 2.2]{PS89}  for a treatment of odd reflections for the periplectic Lie superalgebra. In \cite[Lemma 1]{PS89} the effect on the highest weight of a simple module under odd reflection and inclusion was computed. In combination with Theorem \ref{Thm::2ndmain}, this allows to describe the effect on the highest weight of a tilting module under odd reflection and inclusion, as done in the corollary below. Since every two Borel subalgebras are linked by a series of these two operations, this describes (by iteration) how the highest weights of tilting modules are related under all changes of Borel subalgebras.

\begin{cor} \label{Cor::TiltHtWt} Let $\fb$ and $\fb'$ be two Borel subalgebras, $\alpha$ an odd simple root of $\fb'$ such that either $\Phi(\fb')=\Phi(\fb)\sqcup\{\alpha\}$, or $\Phi(\fb')=(\Phi(\fb)\backslash \{-\alpha\})\sqcup\{\alpha\}$, then
	$$T^{\mf b}(\la)= \begin{cases}
	T^{\mf b'}(\la+\alpha),& \emph{ if } \alpha= 2\epsilon_i. \\
	T^{\mf b'}(\la+\alpha),& \emph{ if } \alpha= \epsilon_i+\epsilon_{i+1} \emph{ with } \la_i-\la_{i+1}\not=0. \\
	T^{\mf b'}(\la+2\alpha),&  \emph{ if } \alpha= \epsilon_i+\epsilon_{i+1} \emph{ with } \la_i-\la_{i+1}=0.
	\end{cases}$$
\end{cor}


\end{document}